\documentclass[a4paper]{amsart}
\usepackage{amsmath}
\usepackage{geometry,amsthm,graphics,tabularx,amssymb,shapepar}
\usepackage{amscd}
\usepackage{paralist}
\usepackage{enumitem}
\usepackage{verbatim}
\usepackage{bm}
\usepackage{comment}
\usepackage[all]{xypic}

\newcommand{\al}{\alpha}

\newcommand{\wh}{\widehat}

\newcommand{\ot}{\otimes}

\newcommand{\C}{\mathbb{C}}

\newcommand{\N}{\mathbb N}
\newcommand{\Z}{\mathbb Z}

\newcommand{\ba}{\begin {eqnarray}}
\newcommand{\ea}{\end {eqnarray}}
\newcommand{\baa}{\begin {eqnarray*}}
\newcommand{\eaa}{\end {eqnarray*}}
\newcommand{\be}{\begin {equation}}
\newcommand{\ee}{\end {equation}}
\newcommand{\bee}{\begin {equation*}}
\newcommand{\eee}{\end {equation*}}

\newcommand{\U}{\mathcal{U}}

\def\({\left(}

\def\){\right)}

\newlength{\dhatheight}

\def \<{{\langle}}
\def \>{{\rangle}}

\theoremstyle{Theorem}

\theoremstyle{Theorem}

\newtheorem{thm}{Theorem}[section]

\newtheorem{lemt}[thm]{Lemma}
\newtheorem{prpt}[thm]{Proposition}

\newtheorem{dfnt}[thm]{Definition}
\numberwithin{equation}{section}
\allowdisplaybreaks

\title[Toroidal Howe dual pairs]{Howe duality in the toroidal setting}
\author{Fulin Chen$^1$}
\address{School of Mathematical Sciences, Xiamen University,
 Xiamen, China 361005}
  \email{chenf@xmu.edu.cn}
\thanks{1. Partially supported by China NSF grant (Nos. 11971397, 12161141001)}

\author{Xin Huang}
\address{School of Mathematical Sciences, Xiamen University,
 Xiamen, China 361005}
 \email{xinhuang2@stu.xmu.edu.cn}

 \author{Shaobin Tan$^2$}
  \address{School of Mathematical Sciences, Xiamen University,
 Xiamen, China 361005}
  \email{tans@xmu.edu.cn}
  \thanks{{3. Partially supported by China NSF grant (No. 12131018)}}

\subjclass[2020]{Primary 17B67}

\date{}
\keywords{Toroidal Lie algebra, quantum torus, Howe duality}
\dedicatory{}
\begin{document}

\begin{abstract}
In this paper, we construct and study various dual pairs acting on the oscillator modules of the
 symplectic toroidal Lie algebras coordinated by irrational quantum tori.
 This extends the classical Howe dual pairs to the toroidal setup.
\end{abstract}

\maketitle

\section{Introduction}
The main goal of this paper is to generalize the classical Howe dual pairs acting
on the oscillator modules of the complex symplectic  Lie algebras to the toroidal setup.
It is well-known (\cite{H}) that there are  two types of irreducible reductive dual pairs in the complex symplectic group $\mathrm{Sp}_{2N}$:
the type I  dual pair $(\mathrm{O}_n,\mathrm{Sp}_{2m})$  and the type II  dual pair $(\mathrm{GL}_n,\mathrm{GL}_m)$,
where  $N,n,m$ are  positive integers such that $N=mn$.
By taking differentials, they lead two  pairs  $(\mathfrak{so}_n,\mathfrak{sp}_{2m})$ and $(\mathfrak{gl}_n, \mathfrak{gl}_{m})$ of
maximally commutating subalgebras  in the
complex symplectic Lie algebra $\mathfrak{sp}_{2N}$.
Denote by $\mathcal{F}_N$ the usual oscillator  $\mathfrak{sp}_{2N}$-module, namely, the space of  polynomial functions on the standard
$\mathfrak{sp}_{2N}$-module $\C^N$.
Via the embedding $\mathfrak{so}_n\hookrightarrow \mathfrak{sp}_{2N}$, the
 oscillator  $\mathfrak{sp}_{2N}$-module $\mathcal{F}_N$ becomes an $\mathfrak{so}_n$-module,
 which can be integrated to an $\mathrm{SO}_n$-module and extended to a locally regular $\mathrm{O}_n$-module.
 Then $(\mathrm{O}_n,\mathfrak{sp}_{2m})$ forms a Howe dual pair on $\mathcal{F}_N$ in the sense that
 \begin{itemize}
 \item each multiplicity space $\mathrm{Hom}_{\mathrm{O}_n}(W_i,\mathcal{F}_N)$ for $i\in I$ is an irreducible $\mathfrak{sp}_{2m}$-module, and
 \item for $i,j\in I$,
$\mathrm{Hom}_{\mathrm{O}_n}(W_i,\mathcal{F}_N)\cong \mathrm{Hom}_{\mathrm{O}_n}(W_j,\mathcal{F}_N)$ if and only if $i=j$,
\end{itemize}
where $\{W_i\}_{i\in I}$ is the set of all non-isomorphic irreducible regular $\mathrm{O}_n$-submodules in $\mathcal{F}_N$.
The above action of the symplectic Lie algebra $\mathfrak{sp}_{2m}$  does not integrate to a
symplectic group representation in general. To obtain an $\mathrm{Sp}_{2m}$-module, one needs to  replace the space of polynomial functions by the
space of Schwartz functions (the so-called  smooth  oscillator module).
On the other hand, via the embedding $\mathfrak{gl}_n\hookrightarrow \mathfrak{sp}_{2N}$, $\mathcal{F}_N$ becomes a $\mathfrak{gl}_n$-module
and integrates to a $\mathrm{GL}_n$-module.
Up to a central action if necessary, $\mathcal{F}_N$ is a locally regular $\mathrm{GL}_n$-module and $(\mathrm{GL}_n,\mathfrak{gl}_m)$ forms a Howe dual pair on $\mathcal{F}_N$.
One may see \cite{H1,H2} for details.

Let $q$ be a generic complex number, and let $\C_q=\C_q[t_0^{\pm1},t_1^{\pm 1}]$ be the quantum torus associated to $q$.
Similar to the construction of finite dimensional classical Lie algebras,
one can define the classical toroidal Lie algebras $\mathfrak{gl}_N(\C_q)$, $\mathfrak{so}_N(\C_q)$ and $\mathfrak{sp}_{2N}(\C_q)$
over $\C_q$.
Following \cite{BGK,CG}, there are one-dimensional non-trivial central extensions,  say
 $\widehat{\mathfrak{gl}}_N(\C_q)$, $\widehat{\mathfrak{so}}_N(\C_q)$ and $\widehat{\mathfrak{sp}}_{2N}(\C_q)$,
 of these classical toroidal Lie algebras.
Motivated by the pairs $(\mathfrak{so}_n,\mathfrak{sp}_{2m})$ and $(\mathfrak{gl}_n, \mathfrak{gl}_{m})$  in $\mathfrak{sp}_{2N}$, we prove in Proposition \ref{prop:dualpairs} that there are three pairs  $(\mathfrak{so}_n,\mathfrak{sp}_{2m}(\C_q))$,
$(\mathfrak{sp}_{2n},\mathfrak{so}_m(\C_q))$  and
$(\mathfrak{gl}_n,\mathfrak{gl}_m(\C_q))$ of maximally commutating subalgebras in $\mathfrak{sp}_{2N}(\C_q)$.
This in turn provides the following pairs of  commutating  subalgebras in $\widehat{\mathfrak{sp}}_{2N}(\C_q)$:
\begin{equation}\label{eq:subalgpairs}
(\mathfrak{so}_n,\widehat{\mathfrak{sp}}_{2m}(\C_q)),\qquad(\mathfrak{sp}_{2n},\widehat{\mathfrak{so}}_m(\C_q)),\qquad
(\mathfrak{gl}_n,\widehat{\mathfrak{gl}}_m(\C_q)).
\end{equation}

The Fock spaces we use to construct various toroidal dual pairs are the oscillator modules $\mathcal{F}_N(\bm{Z})$ for the affine symplectic Lie algebra
constructed in
\cite{FF}, where $\bm{Z}=\Z$ or $\Z+\frac{1}{2}$.
It is remarkable that the affine symplectic Lie algebra action on  $\mathcal{F}_N(\bm{Z})$ can be extended to an action for
the toroidal symplectic Lie algebra $\widehat{\mathfrak{sp}}_{2N}(\C_q)$ (see Proposition \ref{le4}). When $\bm{Z}=\Z$, this
oscillator $\widehat{\mathfrak{sp}}_{2N}(\C_q)$-module $\mathcal{F}_N(\bm{Z})$ was  obtained in \cite{CG}.

Just like the finite dimensional case, via the first pair in \eqref{eq:subalgpairs}, $\mathcal{F}_N(\bm{Z})$ can be integrated to an $\mathrm{SO}_n$-module and extended to a locally
regular $\mathrm{O}_n$-module. Furthermore, $(\mathrm{O}_n,\widehat{\mathfrak{sp}}_{2m}(\C_q))$ forms a Howe dual pair on $\mathcal{F}_N(\bm{Z})$.
Similarly, when $\bm{Z}=\Z$, the $\mathfrak{sp}_{2n}$-action on $\mathcal{F}_N(\bm{Z})$ does not integrate to an $\mathrm{Sp}_{2n}$-module
through the second pair in
\eqref{eq:subalgpairs}.
However, when $\bm{Z}=\Z+\frac{1}{2}$, $\mathcal{F}_N(\bm{Z})$ integrates to a locally regular $\mathrm{Sp}_{2n}$-module and
$(\mathrm{Sp}_{2n},\widehat{\mathfrak{so}}_{m}(\C_q))$ forms a Howe dual pair on $\mathcal{F}_N(\bm{Z})$.
In addition, by using the third pair in \eqref{eq:subalgpairs}, the $\mathfrak{gl}_n$-action on $\mathcal{F}_N(\bm{Z})$ integrates to a group action and,
up to a central action when $\bm{Z}=\Z$,
$(\mathrm{GL}_n,\widehat{\mathfrak{gl}}_m(\C_q))$ forms a Howe dual pair on $\mathcal{F}_N(\bm{Z})$.

As a main result of this paper, we completely determine  the multiplicity-free decomposition of $\mathcal{F}_N(\bm{Z})$ with respect  to the above
Howe
dual pairs (see Theorem \ref{thm:main}).
We also obtain the explicit formula for the joint highest weight vectors of each toroidal dual pair inside $\mathcal{F}_N(\bm{Z})$ and calculate the
corresponding highest
weights (see Propositions \ref{pr:dualpair1}, \ref{pr:dualpair2} and \ref{pr:dualpair3}).
It turns out that all the irreducible regular modules for the complex classical group $\mathrm{O}_n$, $\mathrm{Sp}_{2n}$ or $\mathrm{GL}_n$
are appeared in  $\mathcal{F}_N(\bm{Z})$.
Thus, we present  a duality between all irreducible regular modules for the classical group
$\mathrm{O}_n$, $\mathrm{Sp}_{2n}$ or $\mathrm{GL}_n$ and certain irreducible highest weight modules for the classical
toroidal Lie algebra $\widehat{\mathfrak{sp}}_{2m}(\C_q)$, $\widehat{\mathfrak{so}}_m(\C_q)$ or $\widehat{\mathfrak{gl}}_m(\C_q)$.

Note that the skew Howe duality for general linear and orthogonal Lie groups were generalized to  the toroidal setups in \cite{CLTW} and \cite{CHT}, respectively.
In \cite{CLTW},  by using the reciprocity law attached to the seesaw  pairs, various branching laws for
irreducible integrable highest weight $\widehat{\mathfrak{gl}}_N(\C_q)$-modules were
obtained.
By replacing the duality for $\widehat{\mathfrak{gl}}_N(\C_q)$ in \cite{CLTW} by those toroidal Howe dual pairs given above,
one can also get  various branching laws for certain irreducible  highest weight modules of classical toroidal Lie algebras.
We also remark that the Howe dualities between finite dimensional Lie groups and  completed infinite-rank affine Kac-Moody algebras (or classical Lie subalgebras of
$\mathcal{W}_{1+\infty}$) have been extensively studied (see \cite{FKRW,KR,KWY,W1,W2} for examples).

The paper is organized as follows. In Section 2, we recall the classical toroidal Lie algebras  coordinated by $\mathbb{C}_q$
as well as their one-dimensional central extensions.
In section 3, we review the irreducible regular modules of complex classical groups and introduce a notion of
 irreducible highest weight modules for toroidal Lie algebras.
 In Section 4, we state the main result of this paper, and the proof is presented in Section 5.

For this paper, we work on the field
$\mathbb{C}$ of complex numbers, and we denote by $\mathbb{C}^{\times}$, $\mathbb{Z}$ and $\mathbb{N}$  the sets of nonzero complex numbers, integers and nonnegative integers, respectively.

\section{Toroidal dual pairs}
Throughout this paper, let $N$ be a positive integer.
In this section, we generalize the classical reductive dual pairs in complex symplectic Lie algebra $\mathfrak{sp}_{2N}$  to the toroidal
setting.
\subsection{Reductive dual pairs} We begin by introducing  some standard notations  on the  complex classical groups and Lie algebras for later use.
Namely,
\begin{itemize}
\item the complex general linear group $\mathrm{GL}_{N}:=\{\text{all  $N\times N$ invertible matrices over $\C$}$\},
\item the complex orthogonal group $\mathrm{O}_N:=\{A\in \mathrm{GL}_N \mid J_N A^t J_N=A\}$,
\item the complex special orthogonal group $\mathrm{SO}_N:=\{A\in \mathrm{O}_N\mid \det(A)=1\}$,
\item the complex symplectic group $\mathrm{Sp}_{2N}:=\{A\in \mathrm{GL}_{2N}\mid J_{2N}^{\prime} A^t J_{2N}^{\prime}=A\}$,
\item the complex general linear Lie algebra $\mathfrak{gl}_{N}:=\{\text{all  $N\times N$  matrices over $\C$}$\},
\item the complex orthogonal Lie algebra $\mathfrak{so}_N:=\{A\in \mathfrak{gl}_N \mid J_NA+A^t J_N=0\}$,
\item the complex symplectic Lie algebra $\mathfrak{sp}_{2N}:=\{A\in \mathfrak{gl}_{2N}\mid J_{2N}^{\prime}A+ A^t J_{2N}^{\prime}=0\}$,
\end{itemize}
where $J_N=\sum_{i=1}^{N}E_{i,N+1-i}\in \textrm{GL}_{N}$,
$J_{2N}^{\prime}=\sum_{i=1}^{N}(E_{i,N+i}-E_{N+i,i})\in \textrm{GL}_{2N}$  and $E_{i,j}$'s denote the elementary matrices.

The orthogonal  Lie algebra $\mathfrak{so}_N$ is spanned by the elements
\[e_{i,j}:=E_{i,j}-E_{N+1-j,N+1-i}\]
for $ 1\le i,j\le N$,
and the  symplectic Lie algebra $\mathfrak{sp}_{2N}$ is spanned by the elements
\[f_{i,j}:=E_{i,j}-E_{N+j,N+i},\quad g_{i,j}:=E_{i,N+j}+E_{j,N+i},\quad h_{i,j}:=-E_{N+i,j}-E_{N+j,i}\]
for $ 1\le i,j\le N$.

For a subgroup $B$ of a group $A$, let
\[\mathrm{Comm}_B(A):=\{a\in A\mid ab=ba\ \text{for all}\ b\in B\}\]
be the commutant of $B$ in $A$.
Following \cite{H}, a reductive dual pair $(G,G')$ in $\mathrm{Sp}_{2N}$ is a pair of subgroups $G$ and $G'$ of $\mathrm{Sp}_{2N}$ such that
$G$ and $G'$ are reductive, $\mathrm{Comm}_G(\mathrm{Sp}_{2N})=G'$ and $\mathrm{Comm}_{G'}(\mathrm{Sp}_{2N})=G$.
It is well-known (\cite{H}) that there are exactly two types of irreducible reductive dual pairs in $\mathrm{Sp}_{2N}$:
the type I  dual pair $(\mathrm{O}_n,\mathrm{Sp}_{2m})$  and the type II  dual pair $(\mathrm{GL}_n,\mathrm{GL}_m)$.
Here and henceforth, $m,n$ are two positive
integers such that
\[N=mn.\]

In the language of Lie algebras, this yields the dual pairs  $(\mathfrak{so}_n,\mathfrak{sp}_{2m})$ and $(\mathfrak{gl}_n, \mathfrak{gl}_{m})$ in the
 symplectic Lie algebra $\mathfrak{sp}_{2N}$.
 Here, by a dual pair in a complex Lie algebra $L$, we mean  a pair $(K,K')$ of subalgebras of $L$ such that $\mathrm{Cent}_K(L)=K'$ and $\mathrm{Cent}_{K'}(L)=K$,
where
 \[\mathrm{Cent}_K(L):=\{a\in L\mid [a,b]=0\ \text{for all}\ b\in K\}\]
is the centralizer of $K$ in $L$.
Explicitly, for $1\leq i\leq m$, $1\leq r\leq n$, set
\begin{equation}\label{eq:pi}
 \pi(i,r):=i+(r-1)m.
 \end{equation}
Then $\mathfrak{so}_n$ and $\mathfrak{sp}_{2m}$ are viewed as  subalgebras of $\mathfrak{sp}_{2N}$ through the embeddings
\begin{eqnarray*}
\mathfrak{so}_n\hookrightarrow \mathfrak{sp}_{2N},\quad &&e_{p,s}\mapsto \sum_{k=1}^m f_{\pi(k,p),\pi(k,s)}-f_{\pi(k,n+1-s),\pi(k,n+1-p)},\\
\mathfrak{sp}_{2m}\hookrightarrow \mathfrak{sp}_{2N},\quad &&f_{i,j}\mapsto \sum_{k=1}^n f_{\pi(i,k),\pi(j,k)},\
g_{i,j}\mapsto \sum_{k=1}^n g_{\pi(i,k),\pi(j,n+1-k)},\\
&& h_{i,j}\mapsto \sum_{k=1}^n h_{\pi(i,k),\pi(j,n+1-k)},
\end{eqnarray*}
where $1\le p,s\le n$ and $1\le i,j\le m$.
And,  $\mathfrak{gl}_n$ and $\mathfrak{gl}_{m}$ are  viewed as subalgebras of $\mathfrak{sp}_{2N}$ through the embeddings
\begin{eqnarray*}
&&\mathfrak{gl}_n\hookrightarrow \mathfrak{sp}_{2N},\quad E_{p,s}\mapsto \sum_{k=1}^m f_{\pi(k,p),\pi(k,s)},\\
&&\mathfrak{gl}_m\hookrightarrow \mathfrak{sp}_{2N},\quad E_{i,j}\mapsto \sum_{k=1}^n f_{\pi(i,k),\pi(j,k)},
\end{eqnarray*}
where $1\le p,s\le n$ and $1\le i,j\le m$.

\subsection{Toroidal Lie algebras coordinated by quantum tori}
Throughout this paper, let $q$ be a complex number, which is not a root of unity.
Denote by $\mathbb{C}_q$ the quantum torus associated to $q$. By definition, $\mathbb{C}_q=\mathbb{C}[t_0^{\pm 1},t_1^{\pm 1}]$ as a vector space and the multiplication
is given by
 $$(t_0^{a}t_1^{b})\cdot (t_0^{c}t_1^{d})=q^{bc}t_0^{a+c}t_1^{b+d}$$
for $a,b,c,d\in\mathbb{Z}$. Let $\ \bar{}\ $ be the anti-involution of $\mathbb{C}_q$ determined by
\begin{equation*}
\bar{t}_0=t_0,\quad \bar{t}_1=t_{1}^{-1}.
\end{equation*}
Then we have
\[\overline{t_0^at_1^b}=q^{-ab}t_0^at_1^{-b}\] for $a,b\in\mathbb{Z}$.

We introduce the following toroidal analogies of classical Lie algebras:
\begin{itemize}
\item the general linear toroidal Lie algebra  $\mathfrak{gl}_N(\mathbb{C}_q):=\{(a_{i,j})_{1\leq i,j\leq N}\mid a_{i,j}\in \C_q\}$,
\item the orthogonal toroidal Lie algebra $\mathfrak{so}_N(\mathbb{C}_q):=\{A\in \mathfrak{gl}_N(\mathbb{C}_q)\mid J_NA+\bar{A}^{t}J_N=0\}$,
\item the symplectic toroidal Lie algebra $\mathfrak{sp}_{2N}(\mathbb{C}_q):=\{A\in \mathfrak{gl}_{2N}(\mathbb{C}_q)\mid J_{2N}^{\prime}A+\bar{A}^{t}J_{2N}^{\prime}=0\}$.
\end{itemize}
 Here, for a matrix $A=(a_{ij})\in \mathfrak{gl}_N(\C_q)$, $\bar{A}^t$ stands for the transpose matrix of $\bar{A}=(\overline{a_{ij}})$.
 We will often identify $\mathfrak{gl}_N$ (resp.\,$\mathfrak{so}_N$; resp.\,$\mathfrak{sp}_{2N}$) as a subalgebra of
$\mathfrak{gl}_N(\C_q)$ (resp.\,$\mathfrak{so}_N(\C_q)$; resp.\,$\mathfrak{sp}_{2N}(\C_q)$)
in an obvious way.

Note that the Lie algebra $\mathfrak{so}_N(\mathbb{C}_q)$ is spanned by the elements
\[e_{i,j}(a,b):=E_{i,j}t_0^at_1^b-E_{N+1-j,N+1-i}\overline{t_0^at_1^b}  \]
for $1\leq i,j\leq N, a,b\in\mathbb{Z}$.
Here,  the notation $E_{i,j}t_0^at_1^b$ denotes the matrix whose only nonzero entry is the $(i,j)$-entry which is $t_0^at_1^b$.
Meanwhile, for $1\leq i,j\leq N$, $a,b\in\mathbb{Z}$, form the following elements
\begin{equation*}
  \begin{split}
       &  f_{i,j}(a,b):=E_{i,j}t_0^at_1^b-E_{N+j,N+i}\overline{t_0^at_1^b}, \\
       &  g_{i,j}(a,b):=E_{i,N+j}t_0^at_1^b+E_{j,N+i}\overline{t_0^at_1^b}, \\
       &  h_{i,j}(a,b):=-E_{N+i,j}t_0^at_1^b-E_{N+j,i}\overline{t_0^at_1^b},  \\
  \end{split}
\end{equation*}
which span the Lie algebra $\mathfrak{sp}_{2N}(\mathbb{C}_q)$.

Consider the  $1$-dimensional central extension (cf.\,\cite{BGK})
 \[\widehat{\mathfrak{gl}}_N(\mathbb{C}_q):=\mathfrak{gl}_N(\mathbb{C}_q)\oplus \mathbb{C}\bm{c}\]
of $\mathfrak{gl}_N(\mathbb{C}_q)$, where $\bm c$ is central  and
\begin{eqnarray*}
   &&[E_{i,j}t_0^a t_1^b,E_{k,l}t_0^c t_1^d]\\
   &=&\delta_{j,k}q^{bc}E_{i,l}t_0^{a+c}t_1^{b+d}-\delta_{i,l}q^{ad}E_{k,j}t_0^{a+c}t_1^{b+d}+\frac{1}{2}\delta_{j,k}\delta_{i,l}\delta_{a+c,0}\delta_{b+d,0}q^{bc}a\bm{c}
\end{eqnarray*}
for $1\leq i,j,k,l\leq N$ and $a,b,c,d\in\mathbb{Z}$.
Similarly, we have the $1$-dimensional central extension
\[
  \widehat{\mathfrak{so}}_N(\mathbb{C}_q):=\mathfrak{so}_N(\mathbb{C}_q)\oplus \mathbb{C}\bm c\subset \widehat{\mathfrak{gl}}_N(\mathbb{C}_q)
\]
of $\mathfrak{so}_N(\mathbb{C}_q)$, as well as the $1$-dimensional central extension
\[\widehat{\mathfrak{sp}}_{2N}(\mathbb{C}_q):=\mathfrak{sp}_{2N}(\mathbb{C}_q)\oplus \mathbb{C}\bm c\subset \widehat{\mathfrak{gl}}_{2N}(\mathbb{C}_q)
\] of $\mathfrak{sp}_{2N}(\mathbb{C}_q)$.

\subsection{Toroidal dual pairs}
Motivated by the classical dual pairs in $\mathfrak{sp}_{2N}$, we introduce  the following  Lie algebra embeddings
(into $\mathfrak{sp}_{2N}(\C_q)$):

\begin{equation}\label{eq:dualpair1}\begin{split}&\mathfrak{gl}_n\hookrightarrow \mathfrak{sp}_{2N}(\mathbb{C}_q),\quad
 E_{p,s}\mapsto \sum_{k=1}^{m}f_{\pi(k,p),\pi(k,s)},\\
&\mathfrak{gl}_m(\C_q)\hookrightarrow \mathfrak{sp}_{2N}(\mathbb{C}_q),\quad
 E_{i,j}t_0^a t_1^b\mapsto \sum_{k=1}^{n}f_{\pi(i,k),\pi(j,k)}(a,b),
 \end{split}
 \end{equation}
  \begin{equation}\label{eq:dualpair2}\begin{split}
&\mathfrak{so}_n\hookrightarrow \mathfrak{sp}_{2N}(\mathbb{C}_q),\quad e_{p,s}\mapsto
\sum_{k=1}^{m}(f_{\pi(k,p),\pi(k,s)}-f_{\pi(k,n+1-s),\pi(k,n+1-p)}),\\
&\mathfrak{sp}_{2m}(\C_q)\hookrightarrow \mathfrak{sp}_{2N}(\mathbb{C}_q),\quad \begin{cases}& f_{i,j}(a,b)\mapsto\sum_{k=1}^{n}f_{\pi(i,k),\pi(j,k)}(a,b)\\
& g_{i,j}(a,b)\mapsto \sum_{k=1}^{n}g_{\pi(i,k),\pi(j,n+1-k)}(a,b)\\
& h_{i,j}(a,b)\mapsto \sum_{k=1}^{n}h_{\pi(i,k),\pi(j,n+1-k)}(a,b)
 \end{cases},
 \end{split}
 \end{equation}
 \begin{equation}\label{eq:dualpair3}\begin{split}
&\mathfrak{sp}_{2n}\hookrightarrow \mathfrak{sp}_{2N}(\mathbb{C}_q),\quad \begin{cases}& f_{p,s}\mapsto\sum_{k=1}^{m}f_{\pi(k,p),\pi(k,s)}\\
& g_{p,s}\mapsto \sum_{k=1}^{m}g_{\pi(k,p),\pi(m+1-k,s)}\\
& h_{p,s}\mapsto \sum_{k=1}^{m}h_{\pi(k,p),\pi(m+1-k,s)}
 \end{cases},\\
 &\mathfrak{so}_m(\C_q)\hookrightarrow \mathfrak{sp}_{2N}(\mathbb{C}_q),\quad e_{i,j}(a,b)
\mapsto \sum_{k=1}^{n}(f_{\pi(i,k),\pi(j,k)}(a,b)-q^{-ab}\\
&\qquad\qquad\qquad\qquad\qquad\qquad\qquad\qquad\cdot f_{\pi(m+1-j,k),\pi(m+1-i,k)}(a,-b)),
 \end{split}
 \end{equation}
where $1\le p,s\le n$, $1\le i,j\le m$ and $a,b\in \Z$,
recalling  that $m,n,N$ are positive integers such that $mn=N$.

The following results generalize the classical dual pairs in $\mathfrak{sp}_{2N}$ (see \S\,2.1).

\begin{prpt}\label{prop:dualpairs} Let $N,m,n$ be positive integers such that $mn=N$.
Then, via the embeddings \eqref{eq:dualpair1}-\eqref{eq:dualpair3}, $(\mathfrak{gl}_n,\mathfrak{gl}_m(\C_q))$, $(\mathfrak{so}_n,\mathfrak{sp}_{2m}(\C_q))$ and
$(\mathfrak{sp}_{2n},\mathfrak{so}_m(\C_q))$ are dual pairs in $\mathfrak{sp}_{2N}(\C_q)$.
\end{prpt}

\begin{proof} We will only prove  that, through the embeddings  \eqref{eq:dualpair1},
$(\mathfrak{gl}_n,\mathfrak{gl}_m(\C_q))$ is a dual pair of $\mathfrak{sp}_{2N}(\C_q)$.
The other two cases can be proved in a similar way, and we omit the details.

It is straightforward to see that $\mathfrak{gl}_n$ commutes with $\mathfrak{gl}_{m}(\mathbb{C}_q)$
 (as subalgebras of $\mathfrak{sp}_{2N}(\mathbb{C}_q)$ via \eqref{eq:dualpair1}).
Namely, we have
\begin{equation}\label{eq:glnsubsetcentsp}
\mathfrak{gl}_n\subset \textrm{Cent}_{\mathfrak{sp}_{2N}(\mathbb{C}_q)}(\mathfrak{gl}_{m}(\mathbb{C}_q))\quad\textrm{and}\quad \mathfrak{gl}_{m}(\mathbb{C}_q)\subset \textrm{Cent}_{\mathfrak{sp}_{2N}(\mathbb{C}_q)}(\mathfrak{gl}_n).
\end{equation}
We fix an element
 $$x=\sum_{\substack{1\leq i,j\leq N;\\a,b\in\mathbb{Z}}}(\varphi_{i,j}^{a,b}f_{i,j}(a,b)+\psi_{i,j}^{a,b}g_{i,j}(a,b)+\chi_{i,j}^{a,b}h_{i,j}(a,b))\in \mathfrak{sp}_{2N}(\mathbb{C}_q),$$
where $\varphi_{i,j}^{a,b}, \psi_{i,j}^{a,b}, \chi_{i,j}^{a,b}\in \C$.

If $x\in\textrm{Cent}_{\mathfrak{sp}_{2N}(\mathbb{C}_q)}(\mathfrak{gl}_n)$,
then by comparing the coefficients in the equalities
   $$ [\sum_{k=1}^{m}f_{\pi(k,p),\pi(k,s)}, x]=0\quad (p,s=1,2,\dots,n),$$
we can find that
   $\psi_{i,j}^{a,b}=\chi_{i,j}^{a,b}=0 $ for $1\leq i,j\leq N, ~a,b\in\mathbb{Z},$ and
   $$\varphi_{\pi(k,p),\pi(t,s)}^{a,b}=\delta_{p,s}\varphi_{\pi(k,r),\pi(t,r)}^{a,b}$$
for $1\leq k,t\leq m, ~1\leq p,s,r\leq n,$ and $a,b\in\mathbb{Z}$.
   This together with \eqref{eq:glnsubsetcentsp} gives that
   \[\textrm{Cent}_{\mathfrak{sp}_{2N}(\mathbb{C}_q)}(\mathfrak{gl}_n)= \mathfrak{gl}_{m}(\mathbb{C}_q).\]

Similarly, if  $x\in \textrm{Cent}_{\mathfrak{sp}_{2N}(\mathbb{C}_q)}(\mathfrak{gl}_{m}(\mathbb{C}_q))$,
then by comparing the coefficients in the equalities
 $$[\sum_{k=1}^{n}f_{\pi(i,k),\pi(j,k)}(c,d), x]=0\quad (c,d\in\mathbb{Z},~i,j=1,2,\dots,m,)$$
 we can find that  $\psi_{i,j}^{a,b}=\chi_{i,j}^{a,b}=0$ for $1\leq i,j\leq N, a,b\in\mathbb{Z}$, $\varphi_{\pi(k,p),\pi(t,s)}^{a,b}=0$ for $k\neq t$, $1\leq p,s\leq n$, $a,b\in\mathbb{Z}$,
 and
   $$q^{ad}\varphi_{\pi(k,p),\pi(k,s)}^{a,b}-q^{bc}\varphi_{\pi(t,p),\pi(t,s)}^{a,b}=0$$
   for $1\leq p,s\leq n$, $1\leq k,t\leq m$ and $a,b,c,d\in\mathbb{Z}$.
Since $q$ is not a root of unity, we have $\varphi_{\pi(k,p),\pi(k,s)}^{a,b}=0$ provided that either $a\ne 0$ or $b\ne 0$.
 Together with \eqref{eq:glnsubsetcentsp}, this gives that $\textrm{Cent}_{\mathfrak{sp}_{2N}(\mathbb{C}_q)}(\mathfrak{gl}_{m}(\mathbb{C}_q))= \mathfrak{gl}_n$. Thus we complete the proof.
\end{proof}

\section{Highest weight modules}
In this section, we review the classification of irreducible regular modules for complex classical groups,
and introduce a notion of highest weight module for classical toroidal Lie algebras.

\subsection{Irreducible regular modules for complex classical groups}
In this subsection, we review the classification
of the irreducible (finite dimensional) regular modules for the complex classical groups $\mathrm{GL}_n$, $\mathrm{SO}_n$, $\mathrm{O}_n$, and
$\mathrm{Sp}_{2n}$.
One may see \cite{GW} for details.

Let $\mathrm{H}_{\mathrm{GL}_n}$ be the subgroup of $\mathrm{GL}_n$ consisting of all invertible diagonal matrices,
 and let $\mathrm{N}_{\mathrm{GL}_n}^{+}$ be the subgroup consisting of all unipotent upper-triangular  matrices.
We also write
\begin{eqnarray*} \mathrm{H}_{\mathrm{Sp}_{2n}}=\mathrm{H}_{\mathrm{GL}_{2n}}\cap \mathrm{Sp}_{2n},\qquad
\mathrm{N}_{\mathrm{Sp}_{2n}}^+=\mathrm{N}_{\mathrm{GL}_{2n}}^{+}\cap \mathrm{Sp}_{2n}
\end{eqnarray*}
and
\begin{eqnarray*}
\mathrm{H}_{\mathrm{SO}_{n}}=\mathrm{H}_{\mathrm{GL}_{n}}\cap \mathrm{SO}_{n},\qquad
\mathrm{N}_{\mathrm{SO}_{n}}^+=\mathrm{N}_{\mathrm{GL}_{n}}^{+}\cap \mathrm{SO}_{n}.
\end{eqnarray*}

 Set \[\mathcal{R}(\mathrm{GL}_n):=\{\bm{\mu}=(\mu_1,\mu_2,\dots,\mu_n)\in \Z^n\mid \mu_1\ge \mu_2\ge \cdots\ge \mu_n\}.\]
For every $\bm{\mu}\in \mathcal{R}(\mathrm{GL}_n)$,
denote by $L_{\mathrm{GL}_n}(\bm{\mu})$ the irreducible
regular
$\mathrm{GL}_n$-module generated by a (nonzero) highest weight vector $v_{\bm{\mu}}$ such that
\[\mathrm{N}_{\mathrm{GL}_n}^+.v_{\bm{\mu}}=v_{\bm{\mu}},\qquad h.v_{\bm{\mu}}=h^{\bm{\mu}}v_{\bm{\mu}}\]
for all $h\in \mathrm{H}_{\mathrm{GL}_n}$.
Here and henceforth, for $h=\mathrm{diag}\{h_1,h_2,\dots,h_n\}\in \mathrm{H}_{\mathrm{GL}_n}$ and $\bm{\mu}=
(\mu_1,\mu_2,\dots,\mu_n)\in \Z^n$, we  exploit the notation
\[h^{\bm{\mu}}:=h_1^{\mu_1}h_2^{\mu_2}\cdots h_n^{\mu_n}.\]

Similarly, set
\[ \mathcal{R}(\mathrm{Sp}_{2n}):=\{(\mu_1,\mu_2,\dots,\mu_{2n})\in \Z^{2n}\mid \mu_1\ge \mu_2\ge \cdots\ge \mu_n\ge \mu_{n+1}=\cdots=\mu_{2n}=0\}.\]
For every $\bm{\mu}\in \mathcal{R}(\mathrm{Sp}_{2n})$,
denote by $L_{\mathrm{Sp}_{2n}}(\bm{\mu})$ the irreducible
regular
$\mathrm{Sp}_{2n}$-module generated by a (nonzero) highest weight vector $v_{\bm{\mu}}$ such that
\[\mathrm{N}_{\mathrm{Sp}_{2n}}^+.v_{\bm{\mu}}=v_{\bm{\mu}},\qquad h.v_{\bm{\mu}}=h^{\bm{\mu}}v_{\bm{\mu}}\]
for all $h\in \mathrm{H}_{\mathrm{Sp}_{2n}}$.

For the orthogonal group $\mathrm{O}_n$, write
 \[\mathcal{P}(n):=\{\bm{\mu}=(\mu_1,\mu_2,\dots,\mu_n)\in \N^{n}\mid \mu_1\ge \mu_2\ge \cdots\ge \mu_n\}\]
 for the set of partitions with depth $\le n$.
 For every $\bm{\mu}\in \mathcal{P}(n)$, write
\[d(\bm\mu):=\textrm{max}\{k=1,2,\dots,n\mid\mu_k>0\}\] for the depth of $\bm{\mu}$,
 and write
 \[\bm{\mu}':=(\mu_1',\mu_2',\dots,\mu_n')\] for the transpose of $\bm{\mu}$, where
 $\mu_i^{\prime}=\textrm{Card}\{1\leq k\leq d(\bm\mu)\mid\mu_k\geq i\}$.

Set
\[\mathcal{R}(\mathrm{O}_n):=\{\bm{\mu}\in \mathcal{P}(n)\mid \mu_1'+\mu_2'\le n\}.\]
As usual, write $\lfloor\frac{n}{2}\rfloor$ for  the maximal  integer no larger than $\frac{n}{2}$.
Then we have
\[\mathcal{R}(\mathrm{O}_n)=\{\bm\mu, \tilde{\bm\mu}\mid \bm\mu\in \mathcal{P}(n)\ \text{with}\ d(\bm\mu)\le \lfloor\frac{n}{2}\rfloor\},\]
where  $\tilde{\bm\mu}$ is  the partition obtained from $\bm\mu$ by replacing the first column $\mu_1^{\prime}$ by $n-\mu_1^{\prime}$. Namely,
if $\bm\mu=(\mu_1,\mu_2,\dots,\mu_n)$, then
\[\tilde{\bm\mu}=\{\mu_1,\mu_2,\ldots,\mu_{d(\bm\mu)},1,\dots,1,0,\dots,0\}\]
 with $d(\tilde{\bm\mu})=n-\mu_1^{\prime}$.
 Note that $\bm\mu\ne \tilde{\bm\mu}$ if $d(\bm\mu)<\frac{n}{2}$, and $\bm\mu\ne \tilde{\bm\mu}$ if $d(\bm\mu)=\frac{n}{2}$ with $n$ even.

 We also set
 \[\mathcal{R}(\mathrm{SO}_n):=\{\bm\mu\in \mathcal{P}(n)\mid d(\bm\mu)<\frac{n}{2}\}
 \cup \{\bm{\mu},\overline{\bm{\mu}}\mid \bm\mu\in \mathcal{P}(n)\ \text{with $n$ even and}\ d(\bm\mu)=\frac{n}{2}\},
 \]
 where $\overline{\bm \mu}:=(\mu_1,\mu_2,\ldots,\mu_{\frac{n}{2}-1},-\mu_{\frac{n}{2}})$.
For every $\bm\mu\in \mathcal{R}(\mathrm{SO}_n)$,
denote by $L_{\mathrm{SO}_n}(\bm\mu)$ the irreducible regular $\mathrm{SO}_n$-module
 generated by a (nonzero) highest weight vector $v_{\bm{\mu}}$ such that
\[\mathrm{N}_{\mathrm{SO}_{n}}^+.v_{\bm{\mu}}=v_{\bm{\mu}},\qquad h.v_{\bm{\mu}}=h^{\bm{\mu}}v_{\bm{\mu}} \]
for all $h\in \mathrm{H}_{\mathrm{SO}_{n}}$.

If $n$ is odd, then $\mathrm{O}_n=\textrm{SO}_{n}\times \{\pm I_n\}$, where  $I_n$ denotes the identity matrix of rank $n$.
 Let $L$ be an irreducible regular $\mathrm{O}_n$-module. Then
 \begin{itemize}
 \item $-I_n$ acts as $\pm \mathrm{Id}_{L}$ on $L$; and
  \item there is a (unique) $\bm\mu\in \mathcal{P}(n)$ with $d(\bm\mu)< \frac{n}{2}$ such that $L|_{\mathrm{SO}_n}$ is isomorphic to $L_{\mathrm{SO}_n}(\bm\mu)$.
  \end{itemize}
 We denote $L$ by $L_{\mathrm{O}_n}(\bm\mu)$ if $-I_N$ acts as $\mathrm{Id}_{L}$ on $L$, and
  denote $L$ by $L_{\mathrm{O}_n}(\tilde{\bm\mu})$ if $-I_N$ acts as $-\textrm{Id}_{L}$.

If $n$ is even, then
 $\mathrm{O}_n=\textrm{SO}_{n}\rtimes \{\tau, I_n\}$, where
 $\tau$ denotes the $n\times n$-matrix which interchanges the $\frac{n}{2}$-th row and $(\frac{n}{2}+1)$-th row of identity matrix $I_n$.
  Let $L$ be an  irreducible regular $\textrm{O}_n$-module $L$. Then either
 \begin{itemize}
   \item there is a (unique) $\bm\mu\in \mathcal{P}(n)$ with $d(\bm\mu)=\frac{n}{2}$ such that $L|_{\mathrm{SO}_n}$ is isomorphic to a direct sum of $L_{\mathrm{SO}_n}(\bm \mu)$ and $L_{\mathrm{SO}_n}(\overline{\bm \mu})$; or
   \item there is a (unique) $\bm\mu\in \mathcal{P}(n)$ with $d(\bm\mu)<\frac{n}{2}$ such that $L|_{\mathrm{SO}_n}$  is isomorphic to $L_{\mathrm{SO}_n}(\bm \mu)$, and
    $\tau.v_{\bm\mu}=\pm v_{\bm\mu}$.
 \end{itemize}
In the first case, we denote $L$ by $L_{\mathrm{O}_n}(\bm \mu)$,
and in the second case,  we denote $L$ by $L_{\mathrm{O}_N}(\bm \mu)$ if $\tau.v_{\bm\mu}=v_{\bm\mu}$,
       and denote $L$ by $L_{\mathrm{O}_n}(\tilde{\bm \mu})$ if $\tau.v_{\bm\mu}=-v_{\bm\mu}$.

In summary, let $L$ be an irreducible regular $\mathrm{G}$-module with $\mathrm{G}=\mathrm{GL}_n$, $\mathrm{Sp}_{2n}$, $\mathrm{SO}_n$, or $\mathrm{O}_n$.
Then it is known that there is a (unique) $\bm\mu\in \mathcal{R}(\mathrm{G})$ such that
$L$ is isomorphic to $L_{\mathrm{G}}(\bm\mu)$ (see \cite{GW} for details).
\subsection{A grading on toroidal Lie algebras}
In this subsection, we introduce a canonical grading structure on the (central extensions of) toroidal Lie algebras.

Let $\mathfrak{g}$ be one of the toroidal Lie algebras $\widehat{\mathfrak{gl}}_m(\mathbb{C}_q)$,
$\widehat{\mathfrak{so}}_m(\mathbb{C}_q)$ or $\widehat{\mathfrak{sp}}_{2m}(\mathbb{C}_q)$.
We first endow $\mathfrak{g}$ a canonical $\Z$-grading structure $\mathfrak{g}=\oplus_{a\in \Z}\mathfrak{g}_{(a)}$ with respect to the degree of $t_0$.
Namely,
\begin{eqnarray*}
\widehat{\mathfrak{gl}}_m(\mathbb{C}_q)_{(a)}&=&\text{Span}_\C\{E_{i,j}t_0^at_1^b,\,\delta_{a,0}\bm{c}\mid 1\le i,j\le m, b\in \Z\},\\
\widehat{\mathfrak{so}}_m(\mathbb{C}_q)_{(a)}&=&\text{Span}_\C\{e_{i,j}(a,b),\,\delta_{a,0}\bm{c}\mid 1\le i,j\le m, b\in \Z\},\\
\widehat{\mathfrak{sp}}_{2m}(\mathbb{C}_q)_{(a)}&=&\text{Span}_\C\{f_{i,j}(a,b),\, g_{i,j}(a,b),\, h_{i,j}(a,b),\,\delta_{a,0}\bm{c}\mid 1\le i,j\le m, b\in \Z\}
\end{eqnarray*}
for $a\in \Z$.

Write
\begin{equation*}
\mathfrak{h}:=\begin{cases} \sum_{i=1}^{m}\mathbb{C}E_{i,i}\ &\text{if}\ \mathfrak{g}=\widehat{\mathfrak{gl}}_m(\mathbb{C}_q)\\
\sum_{i=1}^{m}\mathbb{C}e_{i,i}\ &\text{if}\ \mathfrak{g}=\widehat{\mathfrak{so}}_m(\mathbb{C}_q)\\
\sum_{i=1}^{m}\mathbb{C}f_{i,i}\ &\text{if}\ \mathfrak{g}=\widehat{\mathfrak{sp}}_{2m}(\mathbb{C}_q)
\end{cases},
\end{equation*}
which is a finite dimensional abelian subalgebra of $\mathfrak{g}$.
With respect to $\mathfrak{h}$, we have the following root space decomposition of $\mathfrak{g}$:
\[ \mathfrak{g}=\,\oplus_{\alpha\in \mathfrak{h}^*} \mathfrak{g}^\alpha\qquad
(\mathfrak{g}^\alpha:=\{x\in \mathfrak{g}\mid [h,x]=\alpha(h)x,\ \forall\ h\in \mathfrak{h}\}).
\]
Set
\[\Delta(\mathfrak{g}):=\{\alpha\in \mathfrak{h}^*\mid \mathfrak{g}^\alpha\ne 0\},\qquad
Q(\mathfrak{g}):=\Z\Delta(\mathfrak{g})\times \Z.\]
Since the spaces $\mathfrak{g}_{(a)}, a\in \Z$ are all $\mathfrak{h}$-invariant,
$\mathfrak{g}$ admits a  $Q(\mathfrak{g})$-grading as follows:
\begin{eqnarray}\label{eq:Q(g)grading}
\mathfrak{g}=\oplus_{(\alpha,a)\in Q(\mathfrak{g})}\,\mathfrak{g}_{(\alpha,a)}\qquad (\mathfrak{g}_{(\alpha,a)}:=\mathfrak{g}^{\alpha}\cap\mathfrak{g}_{(a)}).
\end{eqnarray}

Explicitly, when $\mathfrak{g}=\widehat{\mathfrak{gl}}_m(\mathbb{C}_q)$, the root system
of $\mathfrak{g}$ is the finite root system
\[\Delta(\mathfrak{g})=\{\epsilon_i-\epsilon_j\mid 1\leq i, j\leq m\}\] of type $A_{m-1}$, and
\begin{eqnarray*}
\widehat{\mathfrak{gl}}_m(\mathbb{C}_q)_{(0,a)}&=&\textrm{Span}_{\mathbb{C}}\{E_{i,i}t_0^at_1^b,~\delta_{a,0}\bm c\mid 1\leq i\leq m,~ b\in\mathbb{Z}\},\\
\widehat{\mathfrak{gl}}_m(\mathbb{C}_q)_{(\epsilon_i-\epsilon_j,a)}&=&\textrm{Span}_{\mathbb{C}}\{E_{i,j}t_0^at_1^b\mid b\in\mathbb{Z}\}
\end{eqnarray*}
for $a\in \Z$, $1\leq i\neq j\leq m$.
Here, $\epsilon_i\in \mathfrak{h}^*$ is defined by
$\epsilon_i(E_{j,j})=\delta_{i,j}$ for $1\le i,j\le m$.

When $\mathfrak{g}=\widehat{\mathfrak{so}}_m(\mathbb{C}_q)$, we define $\epsilon_i\in\mathfrak{h}^{\ast}$ by letting
 $\epsilon_i(e_{j,j})=\delta_{i,j}$ for $1\leq i,j\leq \lfloor\frac{m}{2}\rfloor$.
Then the root system
\begin{equation*}
\Delta(\mathfrak{g})=
  \begin{cases}
    \{\epsilon_i-\epsilon_j, \pm(\epsilon_i+\epsilon_j) \mid 1\leq i, j\leq \frac{m}{2} \} & \mbox{if } m\ \text{is even} \\
    \{\epsilon_i-\epsilon_j, \pm(\epsilon_i+\epsilon_j), \pm \epsilon_i \mid 1\leq i, j\leq \frac{m-1}{2}\} &\mbox{if } m\ \text{is odd}
  \end{cases}
\end{equation*}
of $\mathfrak{g}$ is the finite root system of type $C_{\frac{m}{2}}$ ($m$ even) or of type $BC_{\frac{m-1}{2}}$ ($m$ odd), and
\begin{eqnarray*}
\widehat{\mathfrak{so}}_m(\mathbb{C}_q)_{(0,a)}&=&\textrm{Span}_{\mathbb{C}}\{e_{i,i}(a,b),~\delta_{a,0}\bm c\mid 1\leq i\leq m,~ b\in\mathbb{Z}\},\\
\widehat{\mathfrak{so}}_m(\mathbb{C}_q)_{(\epsilon_i-\epsilon_j,a)}
&=&\textrm{Span}_{\mathbb{C}}\{e_{i,j}(a,b)\mid b\in\mathbb{Z}\},\\
\widehat{\mathfrak{so}}_m(\mathbb{C}_q)_{(\epsilon_k+\epsilon_l,a)}&=&
\textrm{Span}_{\mathbb{C}}\{e_{k,m+1-l}(a,b)\mid b\in\mathbb{Z}\},\\
\widehat{\mathfrak{so}}_m(\mathbb{C}_q)_{(-\epsilon_k-\epsilon_l,a)}&=&
\textrm{Span}_{\mathbb{C}}\{e_{m+1-k,l}(a,b)\mid b\in\mathbb{Z}\},\\
\widehat{\mathfrak{so}}_m(\mathbb{C}_q)_{(\epsilon_i,a)}&=&
    \textrm{Span}_{\mathbb{C}}\{e_{i,\frac{m+1}{2}}(a,b)\mid b\in\mathbb{Z}\},\\
    \widehat{\mathfrak{so}}_m(\mathbb{C}_q)_{(-\epsilon_i,a)}&=&
    \textrm{Span}_{\mathbb{C}}\{e_{m+1-i,\frac{m+1}{2}}(a,b)\mid b\in\mathbb{Z}\}
\end{eqnarray*}
 for $a\in \Z$, $1\leq i\neq j\leq \lfloor\frac{m}{2}\rfloor$ and $1\le k\le l\le \lfloor\frac{m}{2}\rfloor$, where the last two graded subspaces exist only when $m$ is odd.

Finally, when $\mathfrak{g}=\widehat{\mathfrak{sp}}_{2m}(\mathbb{C}_q)$,
define $\epsilon_i\in\mathfrak{h}^{\ast}$ by letting $\epsilon_i(f_{j,j})=\delta_{i,j}$ for $1\leq i,j\leq m$.
Then the root system
\begin{equation*}
\Delta(\mathfrak{g})=    \{\epsilon_i-\epsilon_j, \pm(\epsilon_i+\epsilon_j) \mid 1\leq i, j\leq m \}
\end{equation*}
 of $\mathfrak{g}$ is the finite root system of type $C_{m}$, and
\begin{eqnarray*}
\widehat{\mathfrak{sp}}_{2m}(\mathbb{C}_q)_{(0,a)}&=&\textrm{Span}_{\mathbb{C}}\{f_{i,i}(a,b),~\delta_{a,0}\bm c\mid 1\leq i\leq m,~ b\in\mathbb{Z}\},\\
\widehat{\mathfrak{sp}}_{2m}(\mathbb{C}_q)_{(\epsilon_i-\epsilon_j,a)}&=&\textrm{Span}_{\mathbb{C}}\{f_{i,j}(a,b)\mid b\in\mathbb{Z}\},\\
\widehat{\mathfrak{sp}}_{2m}(\mathbb{C}_q)_{(\epsilon_k+\epsilon_l,a)}&=&\textrm{Span}_{\mathbb{C}}\{g_{k,l}(a,b)\mid b\in\mathbb{Z}\},\\
\widehat{\mathfrak{sp}}_{2m}(\mathbb{C}_q)_{(-\epsilon_k-\epsilon_l,a)}&=&\textrm{Span}_{\mathbb{C}}\{h_{k,l}(a,b)\mid b\in\mathbb{Z}\}
\end{eqnarray*}
 for $a\in \Z$, $1\leq i\neq j\leq m$ and $1\le k\le l\le m$.

\subsection{Highest weight modules for classical toroidal Lie algebras}
Here we introduce a notion of highest weight modules for the classical torodial Lie algebras.

We start with a general definition on highest weight modules.
Let $\Gamma$ be an abelian group equipped with a partial order ``$\succeq$'',
which is additive in the sense that  $\al_1\succeq \beta_1$ and $\al_2\succeq \beta_2$ imply
$\al_1+\al_2\succeq \beta_1+\beta_2$ for $\al_i,\beta_i\in \Gamma$.
By a $(\Gamma,\succeq )$-triangulated Lie algebra, we mean a
$\Gamma$-graded Lie algebra $\mathfrak{g}=\oplus_{\gamma\in\Gamma}\mathfrak{g}_{\gamma}$ such that $\mathfrak{g}_0$ is abelian and
\begin{align}\label{cgtri}
\mathfrak{g}=\mathfrak{g}_{+}\oplus \mathfrak{g}_{0}\oplus \mathfrak{g}_{-}\quad (\mathfrak g_{\pm}:=\oplus_{\pm \gamma\succ 0}\mathfrak g_{\gamma}).
\end{align}

\begin{dfnt}
Let $\mathfrak{g}$ be a  $(\Gamma,\succeq )$-triangulated Lie algebra. A $\mathfrak{g}$-module
$W$ is called a highest weight module with highest weight $\lambda\in \mathfrak{g}_{ 0}^*$ if there exists a nonzero vector $v_{\lambda}\in W$ such that
\begin{equation*}
W=\mathcal{U}(\mathfrak{g})v_{\lambda},\qquad \mathfrak{g}_{+}.v_{\lambda}=0,\qquad h.v_{\lambda}=\lambda(h)v_{\lambda}
\end{equation*}
for all $h\in \mathfrak{g}_{0}$.
\end{dfnt}

Given a $\lambda\in \mathfrak{g}_{0}^*$. Let $\C v_\lambda$ be a one-dimensional
$(\mathfrak{g}_{+}\oplus \mathfrak{g}_{0})$-module on which $\mathfrak{g}_{+}$ acts trivially and $h.v_{\lambda}=\lambda(h)v_{\lambda}$ for all
$h\in \mathfrak{g}_{0}$.
Form an induced $\mathfrak{g}$-module
\begin{align}\label{defvgu}
 V_{\mathfrak{g},\Gamma,\succeq}(\lambda):=\U(\mathfrak{g})\ot_{\U(\mathfrak{g}_{+}\oplus\mathfrak{g}_{0})} \C v_\lambda.
 \end{align}
Note that there is a natural $\Gamma$-grading structure on  $V_{\mathfrak{g},\Gamma,\succeq}(\lambda)$ defined  by letting $\deg v_{\lambda}=0$.
Let $J_{\mathfrak{g},\Gamma,\succeq}(\lambda)$ be the (unique) maximal proper $\Gamma$-graded submodule of $V_{\mathfrak{g},\Gamma,\succeq}(\lambda)$.
Set
\begin{align}
L_{\mathfrak{g},\Gamma,\succeq}(\lambda):=V_{\mathfrak{g},\Gamma,\succeq}(\lambda)/J_{\mathfrak{g},\Gamma,\succeq}(\lambda),
\end{align}
which  is a graded irreducible  $\mathfrak{g}$-module.
The following result shows that  $L_{\mathfrak{g},\Gamma,\succeq}(\lambda)$ is in fact an irreducible $\mathfrak{g}$-module.

\begin{lemt}\label{Lirr}
For any $\lambda\in\mathfrak{g}_0^{\ast}$, the $\mathfrak{g}$-module $L_{\mathfrak{g},\Gamma,\succeq}(\lambda)$ is irreducible.
\end{lemt}

\begin{proof}
Let $v=v_1+\cdots+v_k$ be a nonzero vector in $L_{\mathfrak{g},\Gamma,\succeq}(\lambda)$ with $\deg v_{i}=\gamma_i\preceq 0$ for some distinct $\gamma_i\in\Gamma$.
For each $1\le i\le k$, $\U(\mathfrak{g})v_i$ is a nonzero graded $\mathfrak{g}$-submodule of
$L_{\mathfrak{g},\Gamma,\succeq}(\lambda)$ and hence $\U(\mathfrak{g})v_i=L_{\mathfrak{g},\Gamma,\succeq}(\lambda)$. Consequently,
there exists
$x_i\in\mathfrak{g}_{-\gamma_i}$ such that $ x_i. v_i=v_{\lambda}$.
Let $\gamma_j$ be a minimal element in the set $\{\gamma_1,\dots,\gamma_k\}$.
As $\deg x_j.v_i=\gamma_i-\gamma_j$, we see that $x_j. v_i\ne 0$ if and only if $\gamma_i=\gamma_j$.
Thus we have $x_j. v=x_j. v_j=v_{\lambda}$.
This gives that $\U(\mathfrak{g})v=\U(\mathfrak{g})v_{\lambda}=L_{\mathfrak{g},\Gamma,\succeq}(\lambda)$, as desired.
\end{proof}

Lemma $\ref{Lirr}$ implies that any irreducible highest weight $\mathfrak{g}$-module with highest weight $\lambda\in\mathfrak{g}_{0}^{\ast}$
must be isomorphic to $L_{\mathfrak{g},\Gamma,\succeq}(\lambda)$.
We also notice that, for $\lambda,\mu\in\mathfrak{g}_0^{\ast}$,   $L_{\mathfrak{g},\Gamma,\succeq}(\lambda)\cong L_{\mathfrak{g},\Gamma,\succeq}(\mu)$
 if and only if $\lambda=\mu$.

Now let $\mathfrak{g}$ be one of the toroidal Lie algebras $\widehat{\mathfrak{gl}}_m(\mathbb{C}_q)$,
$\widehat{\mathfrak{so}}_m(\mathbb{C}_q)$ or $\widehat{\mathfrak{sp}}_{2m}(\mathbb{C}_q)$.
Recall from \eqref{eq:Q(g)grading} that  $\mathfrak{g}$ is a $Q(\mathfrak{g})$-graded Lie algebra.
Choose a simple root base $\Pi(\mathfrak{g})$ of the finite irreducible root system $\Delta(\mathfrak{g})$.
Then there is a canonical  partial order $\succeq_{\Pi(\mathfrak{g})}$ on $\Z\Delta(\mathfrak{g})$ such that
\[\alpha \succeq_{\Pi(\mathfrak{g})} \beta \iff \alpha-\beta\in \sum_{\gamma\in \Pi(\mathfrak{g})} \N \gamma
\]
for $\alpha,\beta\in \Z\Delta(\mathfrak{g})$.
Extend $\succeq_{\Pi(\mathfrak{g})}$ to a partial order on $Q(\mathfrak{g})$ such that
$$(\alpha,a)\succeq_{\Pi(\mathfrak{g})} (\beta,b)\iff\text{either}\ a>b, \textrm{ or }a=b \textrm{ and }\alpha\succeq_{\Pi(\mathfrak{g})}  \beta$$
 for $\alpha,\beta\in \Z\Delta(\mathfrak{g}), a,b\in \Z$.
 Then $\mathfrak{g}$ becomes a $(Q(\mathfrak{g}),\succeq_{\Pi(\mathfrak{g})})$-triangulated Lie algebra.
 In particular, for every $\lambda\in \mathfrak{g}_0^\ast$, we have an irreducible
 highest weight $\mathfrak{g}$-module
 \begin{eqnarray}
 \label{eq:LgQg}
 L_{\mathfrak{g},Q(\mathfrak{g}),\succeq_{\Pi(\mathfrak{g})}}(\lambda).
 \end{eqnarray}

\section{Howe dual pairs on toroidal oscillator modules}
In this section, we state the main result of this paper.

 \subsection{The oscillator $\widehat{\mathfrak{sp}}_{2N}(\C_q)$-module $\mathcal{F}_N(\bf{Z})$}
 Let $(\mathfrak{a},\<\cdot,\cdot\>)$ be a symplectic complex vector space of dimension $2N$,
and let $\mathcal{W}_N$ be the Weyl algebra associated to $(\mathfrak{a},\<\cdot,\cdot\>)$. Recall that $\mathcal{W}_N$ is the unital associative algebra
generated by the space $\mathfrak{a}$ and satisfies  the relations
\[ uv-vu=\<u,v\>\,1\]
for $u,v\in \mathfrak{a}$.
Fix a
polarization
\[\mathfrak{a}=\mathfrak{a}_1\oplus\mathfrak{a}_2\quad (\mathfrak{a}_1\cong \C^N\cong \mathfrak{a}_2)\] of $\mathfrak{a}$  into maximal isotropic subspaces.
Let $\mathcal{F}_N$ be the simple $\mathcal{W}_N$-module generated by a vacuum vector $|0\>$ such that $\mathfrak{a}_2.|0\>=0$.
It is known that the quadratic elements $u v$ $(u,v\in \mathfrak{a})$ in $\mathrm{End}(\mathcal{F}_N)$ span the symplectic Lie algebra
$\mathfrak{sp}_{2N}$, and so $\mathcal{F}_N$ becomes an $\mathfrak{sp}_{2N}$-module.

 In what follows we generalize the  oscillator $\mathfrak{sp}_{2N}$-module $\mathcal{F}_N$ to the  symplectic  toroidal  Lie algebra
$\widehat{\mathfrak{sp}}_{2N}(\C_q)$ (cf.\cite{CG}).

As in \cite{FF}, we first introduce an enlarge Weyl algebra $W_N(\bm{Z})$, where $\bm{Z}=\Z$ or $\Z+\frac{1}{2}$.
By definition, it is generated by the elements  $u(k)$ for $u\in \mathfrak{a}$, $k\in \bm{Z}$, and
subject to the relations
\[u(k)v(r)-v(r)u(k)=\<u,v\>\,\delta_{k+r,0}\,1\]
for  $u,v\in \mathfrak{a},\ k,r\in \bm{Z}$.
Define a  normal ordering on $\mathcal{W}_N(\bm{Z})$ as follows:
\begin{equation}\label{no}
:u(k)v(r):=\begin{cases}
u(k)v(r)\quad  &\textrm{if}\ r>0\\
\frac{1}{2}(u(k)v(r)+v(r)u(k))\quad &\text{if}\ r=0\\
v(r)u(k)\quad  &\textrm{if}\ r<0
\end{cases},
\end{equation}
where $u,v\in \mathfrak{a}$ and $k,r\in \bm{Z}$.

Let $\mathcal{F}_N(\bm{Z})$ be the simple $\mathcal{W}_N(\bm{Z})$-module generated by a (nonzero) vacuum vector $|0\>$ such that
\begin{itemize}
\item $u(k).|0\>=0$   for all $u\in \mathfrak{a}$
and $k\in \bm{Z}$ with $k>0$;  and
\item  $v(0).|0\>=0$ for all $v\in \mathfrak{a}_2$ if $\bm{Z}=\Z$.
\end{itemize}
Note that $\mathcal{F}_N(\bm{Z})$ is a free $\mathcal{W}_N(\bm{Z})^-$-module of rank $1$, where
$\mathcal{W}_N(\bm{Z})^-$ denotes the subalgebra of $\mathcal{W}_N(\bm{Z})$ generated by the elements $u(k)$ for  $u\in \mathfrak{a}$
and $k\in \bm{Z}$ with $k<0$, and the elements $v(0)$ for $v\in \mathfrak{a}_1$ if $\bm{Z}=\Z$.

Fix a basis $\psi_1,\psi_2,\dots,\psi_N$ of $\mathfrak{a}_1$, and let $\bar\psi_1,\bar\psi_2,\dots,\bar\psi_N$ be the dual basis in $\mathfrak{a}_2$
such that $\<\bar\psi_i,\psi_j\>=\delta_{i,j}$ for $1\le i,j\le N$.
For $1\leq r\leq m$ and $1\leq k\leq n$, we will also write (see \eqref{eq:pi})
\[\psi_r^k:=\psi_{\pi(r,k)},\qquad \bar{\psi}_r^k:=\bar{\psi}_{\pi(r,k)}.\]

For every $b\in\mathbb{Z}$, set
\begin{equation*}\label{}
\omega_{\bm{Z}}(b)=
  \begin{cases}
    0, & \mbox{if } b=0 \\
    \frac{1}{2}\frac{q^b+1}{q^b-1}, & \mbox{if }b\neq 0\ \text{and}\ \bm{Z}=\Z\\
    \frac{q^{\frac{1}{2}b}}{q^b-1}, & \mbox{if }b\neq 0\ \text{and}\ \bm{Z}=\Z+\frac{1}{2}
  \end{cases},
\end{equation*}
where $q^{\frac{1}{2}b}=\left(q^{\frac{1}{2}}\right)^b$ and $q^{\frac{1}{2}}$ is a
 square root of $q$ (fixed throughout this paper).

Define
\[\rho_{\bm{Z}}:\ \widehat{\mathfrak{sp}}_{2N}(\mathbb{C}_q)\longrightarrow  \textrm{End}(\mathcal{F}_N(\bm{Z}))\]
 to be a linear map such that $\rho(\bm c)=-1$ and
\begin{equation}\label{ho}
\begin{split}
 \rho_{\bm Z}(f_{ij}(a,b))&=\sum_{r\in\bm{Z}}q^{-br}:\psi_{i}(a-r)\bar{\psi}_{j}(r):+\delta_{i,j}\delta_{a,0}\omega_{\bm{Z}}(b)\bm{c},\\
   \rho_{\bm Z}(g_{ij}(a,b))&=\sum_{r\in\bm{Z}}q^{-br}:\psi_{i}(a-r)\psi_{j}(r):,\\
   \rho_{\bm Z}(h_{ij}(a,b))&=\sum_{r\in\bm{Z}}q^{-br}:\bar\psi_{i}(a-r)\bar\psi_{j}(r):,\\
\end{split}
\end{equation}
where $1\leq i,j\leq N$ and $a,b\in\mathbb{Z}$.

Then we have the following result.

\begin{prpt}\label{le4}
  The linear map $\rho_{\bm{Z}}$  gives a $\widehat{\mathfrak{sp}}_{2N}(\mathbb{C}_q)$-module structure on $\mathcal{F}_N(\bm{Z})$.
\end{prpt}
\begin{proof} When $\bm{Z}=\Z$, the proposition was proved in \cite[Theorem 2.1]{CG}, and the case that $\bm{Z}=\Z+\frac{1}{2}$ can be
proved in a similar way.
\end{proof}
\subsection{Some highest weight modules}
Let $\mathfrak{g}$ be one of the toroidal Lie algebras $\widehat{\mathfrak{gl}}_m(\mathbb{C}_q)$, $\widehat{\mathfrak{sp}}_{2m}(\mathbb{C}_q)$ or
$\widehat{\mathfrak{so}}_m(\mathbb{C}_q)$.
In this subsection, we introduce some special highest weight $\mathfrak{g}$-modules, which are concerned about in this paper.

In what follows we define a simple  root base $\Pi(\mathfrak{g},\bm{Z})$ of $\Delta(\mathfrak{g})$ ($\bm{Z}=\Z+\frac{1}{2}$ if
$\mathfrak{g}=\widehat{\mathfrak{so}}_m(\mathbb{C}_q)$), and give the associated triangular decomposition
$\mathfrak{g}=\mathfrak{g}_{+,\bm{Z}}\oplus \mathfrak{g}_0\oplus \mathfrak{g}_{-,\bm{Z}}$ of $\mathfrak{g}$, where
  $\mathfrak{g}_{\pm,\bm{Z}}=\oplus_{\pm (\alpha,a)\succ_{\Pi(\mathfrak{g},\bm{Z})} 0}\mathfrak{g}_{(\alpha,a)}$.
 When $\mathfrak{g}=\widehat{\mathfrak{gl}}_m(\mathbb{C}_q)$, the simple root base
 \begin{equation}\label{eq:Pi1}
 \Pi(\mathfrak{g},\bm{Z})=\{\epsilon_1-\epsilon_2,\epsilon_2-\epsilon_3,\dots,\epsilon_{m-1}-\epsilon_m\}
 \end{equation}
 for both $\bm{Z}=\Z+\frac{1}{2}$ and $\Z$,
and we have
 \begin{eqnarray*}
\widehat{\mathfrak{gl}}_m(\mathbb{C}_q)_{+,\bm{Z}}&=&\sum_{1\le i,j\le m; a,b\in \Z, a>0} \mathbb{C}E_{i,j}t_0^at_1^b+\sum_{
1\le i<j\le m; b\in \Z}\mathbb{C}E_{i,j}t_1^b,\\
\widehat{\mathfrak{gl}}_m(\mathbb{C}_q)_{-,\bm{Z}}&=&\sum_{1\le i,j\le m; a,b\in \Z, a<0} \mathbb{C}E_{i,j}t_0^at_1^b+\sum_{
1\le i>j\le m; b\in \Z}\mathbb{C}E_{i,j}t_1^b,\\
\widehat{\mathfrak{gl}}_m(\mathbb{C}_q)_{0}&=&\sum_{1\leq i\leq m; b\in\mathbb{Z}}\mathbb{C}(E_{i,i}t_1^b)+\mathbb{C}\bm{c},
\end{eqnarray*}

When $\mathfrak{g}=\widehat{\mathfrak{sp}}_{2m}(\mathbb{C}_q)$, the simple root bases
\begin{eqnarray}
\label{eq:Pi3} \Pi(\mathfrak{g},\Z+\frac{1}{2})&=&\{\epsilon_i-\epsilon_{i+1},~2\epsilon_{m}\mid 1\leq i\leq m-1\},\\
\label{eq:Pi4} \Pi(\mathfrak{g},\Z)&=&\{-2\epsilon_{1},~\epsilon_i-\epsilon_{i+1}\mid 1\leq i\leq m-1\}.
\end{eqnarray}
And, we have that
\begin{equation*}
  \begin{split}
&  \widehat{\mathfrak{sp}}_{2m}(\mathbb{C}_q)_{\pm,\Z+\frac{1}{2}}=\widehat{\mathfrak{gl}}_{2m}(\mathbb{C}_q)_{\pm}\cap
\widehat{\mathfrak{sp}}_{2m}(\mathbb{C}_q),\qquad \widehat{\mathfrak{sp}}_{2m}(\mathbb{C}_q)_{0}=\widehat{\mathfrak{gl}}_{2m}(\mathbb{C}_q)_{0}\cap
\widehat{\mathfrak{sp}}_{2m}(\mathbb{C}_q),\\
  &  \widehat{\mathfrak{sp}}_{2m}(\mathbb{C}_q)_{+,\Z}=\sum_{1\le i,j\le m; a,b\in \Z, a>0}\mathbb{C} f_{i,j}(a,b)+\sum_{1\le i<j\le m;b\in\Z} \mathbb{C}f_{i,j}(0,b) \\
   &+\sum_{1\le i,j\le m; a,b\in \Z, a>0}\mathbb{C}g_{i,j}(a,b)+\sum_{1\le i,j\le m; a,b\in \Z, a>0}\mathbb{C}h_{i,j}(a,b)+\sum_{1\le i, j\le m; b\in \Z}\mathbb{C}h_{i,j}(0,b),\\
    &\widehat{\mathfrak{sp}}_{2m}(\mathbb{C}_q)_{-,\Z}=\sum_{1\le i,j\le m; a,b\in \Z, a<0}\mathbb{C} f_{i,j}(a,b)+\sum_{1\le i>j\le m;b\in\Z} \mathbb{C}f_{i,j}(0,b) \\
   &+\sum_{1\le i,j\le m; a,b\in \Z, a<0}\mathbb{C}g_{i,j}(a,b)+\sum_{1\le i,j\le m; a,b\in \Z, a<0}\mathbb{C}h_{i,j}(a,b)+\sum_{1\le i, j\le m; b\in \Z}\mathbb{C}g_{i,j}(0,b).\\
  \end{split}
\end{equation*}

When $\mathfrak{g}=\widehat{\mathfrak{so}}_m(\mathbb{C}_q)$ and $\bm{Z}=\Z+\frac{1}{2}$, the simple base
 \begin{equation}\label{eq:Pi2}
 \Pi(\mathfrak{g},\bm{Z})=
   \begin{cases}
     \{\epsilon_i-\epsilon_{i+1},~2\epsilon_{\frac{m}{2}},\mid 1\leq i\leq \frac{m}{2}-1\} & \mbox{if } m\ \text{is even} \\
     \{\epsilon_i-\epsilon_{i+1},~\epsilon_{\frac{m-1}{2}},\mid 1\leq i\leq \frac{m-1}{2}-1\} & \mbox{if } m\ \text{is odd}
   \end{cases}.
 \end{equation}
 In this case, we have
 \[\widehat{\mathfrak{so}}_m(\mathbb{C}_q)_{\pm,\bm{Z}}=
 \widehat{\mathfrak{gl}}_m(\mathbb{C}_q)_{\pm,\bm{Z}} \cap \widehat{\mathfrak{so}}_m(\mathbb{C}_q),\quad
 \widehat{\mathfrak{so}}_m(\mathbb{C}_q)_{0}=
 \widehat{\mathfrak{gl}}_m(\mathbb{C}_q)_{0} \cap \widehat{\mathfrak{so}}_m(\mathbb{C}_q).
 \]

Next we introduce some linear functionals $\eta_{\bm\mu,\bm{Z}}$ on $\mathfrak{g}_0$.
Let $i,\ell$ be two positive integers such that $1\leq i\leq m$ and $1\leq \ell\leq n$.
Set
\begin{equation*}
J_{i,\ell}^{\bm Z}:= \{r\in\bm Z\mid 1\leq (r-\frac{1}{2}+\epsilon)m+i\leq \ell\},
\end{equation*}
\begin{equation*}
\bar{J}_{i,\ell}^{\bm Z}:=\{r\in\bm Z\mid 1\leq (-r+\frac{1}{2}-\epsilon)m-i+1\leq \ell\},
\end{equation*}
and
\begin{equation*}
\bar{\ell}_1:=
\begin{cases}
  \frac{\ell_1}{2}m+\ell_2 & \mbox{if } \ell_1 \textrm{ is even} \\
  \frac{\ell_1+1}{2}m & \mbox{if } \ell_1 \textrm{ is odd}
\end{cases},\qquad
\bar{\ell}_2:=
\begin{cases}
  \frac{\ell_1}{2}m & \mbox{if } \ell_1 \textrm{ is even}\\
  \frac{\ell_1-1}{2}m+\ell_2 & \mbox{if } \ell_1 \textrm{ is odd}
\end{cases},
\end{equation*}
where $\ell_1,\ell_2$ are the nonnegative integers such that
\[\ell=\ell_1m+\ell_2,\quad \ell_1
\in\mathbb{N},\quad 1\leq \ell_2\leq m.\]
We also set
  \begin{equation}\label{eq:defepsilon}
\epsilon:=
  \begin{cases}
    0 & \mbox{if } \bm Z=\mathbb{Z}+\frac{1}{2} \\
    \frac{1}{2} & \mbox{if }\bm{Z}=\mathbb{Z}
  \end{cases}.
\end{equation}

When $\mathfrak{g}=\widehat{\mathfrak{gl}}_m(\mathbb{C}_q)$, for every $\bm\mu=(\mu_1,\mu_2,\dots,\mu_n)\in \mathcal{R}(\mathrm{GL}_n)$ with
\begin{equation}
\label{eq:muformGLn}\mu_1\ge\mu_2\ge\cdots\ge\mu_p>\mu_{p+1}=\cdots=\mu_{s-1}=0>\mu_{s}\ge\cdots\ge\mu_n,\end{equation}
we define $\eta_{\bm\mu,\bm Z}$ to be the linear functional on $\mathfrak{g}_{0}$ such that $\eta_{\bm\mu,\bm Z}(\bm c)=-n$ and
\begin{equation*}
\begin{split}
 &\eta_{\bm\mu,\bm Z}(E_{i,i}t_1^b)\\
 =&\sum_{k=1}^{p-1}(\mu_k-\mu_{k+1})\sum_{r\in J_{i,k}^{\bm Z}}q^{-br}
+\mu_p\sum_{r\in J_{i,p}^{\bm Z}}q^{-br}+\mu_{s}\sum_{r\in\bar{J}_{i,n-s+1}^{\bm Z}}q^{-br}\\
&-\sum_{k=s+1}^{n}(\mu_{k-1}-\mu_{k})\sum_{r\in\bar{J}_{i,n-k+1}^{\bm Z}}q^{-br}
+\epsilon n-n\omega_{\bm Z}(b),
\end{split}
\end{equation*}
where $1\le i\le m$ and $b\in \Z$.

When $\mathfrak{g}=\widehat{\mathfrak{sp}}_{2m}(\C_q))$, for every $\bm\mu=(\mu_1,\mu_2,\dots,\mu_n)\in \mathcal{R}(\mathrm{O}_{n})$,
 define the linear functional $\eta_{\bm\mu,\bm{Z}}$ on $\mathfrak{g}_0$
by letting $\eta_{\bm\mu,\bm{Z}}(\bm{c})=-n$ and
\begin{eqnarray*}
\eta_{\bm\mu,\bm Z}(f_{i,i}(0,b))=\sum_{k=1}^{n}(\mu_k-\mu_{k+1})(\sum_{r\in J_{i,\bar{k}_1}^{\bm Z}}q^{-br}-\sum_{r\in \bar{J}_{i,\bar{k}_2}^{\bm Z}}q^{-br})
+\epsilon n-n\omega_{\bm Z}(b),
\end{eqnarray*}
where $1\le i\le m$, $b\in \Z$ and $\mu_{n+1}=0$.

When $\mathfrak{g}=\widehat{\mathfrak{so}}_m(\C_q)$, for every $\bm\mu=(\mu_1,\mu_2,\dots,\mu_{2n})\in \mathcal{R}(\mathrm{Sp}_{2n})$ with  $\bm{Z}=\mathbb{Z}+\frac{1}{2}$, define
$\eta_{\bm\mu,\bm{Z}}$ to be the linear functional on $\mathfrak{g}_0$ such that $\eta_{\bm\mu,\bm Z}(\bm c)=-n$ and
\begin{equation*}
\begin{split}
 \eta_{\bm\mu,\bm Z}(e_{i,i}(0,b))
 =\sum_{k=1}^{n}(\mu_k-\mu_{k+1})\sum_{r\in J_{i,k}^{\bm Z}}q^ {-br}-\sum_{k=1}^{n}(\mu_k-\mu_{k+1})\sum_{r\in J_{m+1-i,k}^{\bm Z}}q^{br}-2n\omega_{\bm Z}(b),
\end{split}
\end{equation*}
where $1\le i\le m$ and $b\in \Z$.

For $\lambda\in \mathfrak{g}_0^*$, we  set (see \eqref{eq:LgQg})
\[ L_{\mathfrak{g},\bm{Z}}(\lambda):=L_{\mathfrak{g},Q(\mathfrak{g}),\succeq_{\Pi(\mathfrak{g},\bm{Z})}}(\lambda).
\]
In this notation,  we have the following irreducible highest weight modules for toroidal Lie algebras $\widehat{\mathfrak{gl}}_m(\mathbb{C}_q)$,
$\widehat{\mathfrak{sp}}_{2m}(\mathbb{C}_q)$
 or $\widehat{\mathfrak{so}}_m(\mathbb{C}_q)$, which are needed in next subsection:
\begin{itemize}
\item $L_{\widehat{\mathfrak{gl}}_m(\mathbb{C}_q),\bm{Z}}(\eta_{\bm\mu,\bm{Z}})$\quad for all\ $\bm\mu\in \mathcal{R}(\mathrm{GL}_n)$;
\item $L_{\widehat{\mathfrak{sp}}_{2m}(\mathbb{C}_q),\bm{Z}}(\eta_{\bm\mu,\Z})$\quad for all\ $\bm\mu\in \mathcal{R}(\mathrm{O}_n)$;
\item $L_{\widehat{\mathfrak{so}}_m(\mathbb{C}_q),\Z+\frac{1}{2}}(\eta_{\bm\mu,\Z+\frac{1}{2}})$\quad for all\ $\bm\mu\in \mathcal{R}(\mathrm{Sp}_{2n})$.
\end{itemize}

\subsection{Main results}
We start with the definition of Howe dual pairs (\cite{GW}).
Recall that a module for a complex reductive group is called  locally regular if it is a direct sum of irreducible regular submodules.

\begin{dfnt} Let $\mathrm{G}$ be a complex reductive group,  $\mathcal{L}$  a complex Lie algebra, and $W$  a complex vector space.

$\mathrm{(i)}$ We say that $W$ is a locally regular $(\mathrm{G},\mathcal{L})$-module if there is a $\mathrm{G}$-module action as well as
an $\mathcal{L}$-module action on it such that these two actions commute, and as a $\mathrm{G}$-module, $W$ is locally regular.

$\mathrm{(ii)}$ We say that $(\mathrm{G}, \mathcal{L})$ is a Howe dual pair on $W$ if
\begin{itemize}
  \item $W$ is a locally regular $(\mathrm{G},\mathcal{L})$-module;
  \item for every $i\in I$, the $\mathcal{L}$-submodule $\textrm{Hom}_\mathrm{G}(W_i,W)$ of $W$ is irreducible; and
  \item for any  $i, j \in I$, $\textrm{Hom}_\mathrm{G}(W_i,W)\cong \textrm{Hom}_\mathrm{G}(W_j,W)$ as $\mathcal{L}$-modules if and only if $i = j$,
\end{itemize}
where
$W_i, i\in I$ exhaust all non-isomorphic irreducible regular $\mathrm{G}$-submodules of $W$.
\end{dfnt}

For example, recall the classical dual pairs $(\mathfrak{gl}_n,\mathfrak{gl}_m)$ and $(\mathfrak{so}_n,\mathfrak{sp}_{2m})$ in $\mathfrak{sp}_{2N}$.
By taking restriction, the oscillator $\mathfrak{sp}_{2N}$-module $\mathcal{F}_N$ becomes an $\mathfrak{so}_n\oplus \mathfrak{sp}_{2m}$-module.
As an $\mathfrak{so}_n$-module, $\mathcal{F}_N$ can be integrated to an $\mathrm{SO}_n$-module and extended to a (locally regular) $\mathrm{O}_n$-module.
Then  $(\mathrm{O}_n,\mathfrak{sp}_{2m})$ forms a Howe dual pair on $\mathcal{F}_N$.
We remark that, as an $\mathfrak{sp}_{2m}$-module, $\mathcal{F}_N$ can not be integrated to an $\mathrm{Sp}_{2m}$-module.

Similarly, by taking restriction, $\mathcal{F}_N$ becomes a  $(\mathfrak{gl}_n\oplus \mathfrak{gl}_m)$-module.
  As a $\mathfrak{gl}_n$-module, $\mathcal{F}_N$ can be integrated  to a $\mathrm{GL}_n$-module.
  Up to a central action of $\mathrm{GL}_n$ if necessary, $(\mathrm{GL}_n,\mathfrak{gl}_m)$  becomes a Howe dual pair on $\mathcal{F}_N$ as well.

Now we turn to consider the toroidal Lie algebras case.
Recall the three dual pairs $(\mathfrak{gl}_n,\mathfrak{gl}_m(\C_q))$, $(\mathfrak{so}_n,\mathfrak{sp}_{2n}(\C_q))$ and
$(\mathfrak{sp}_{2n},\mathfrak{so}_m(\C_q))$  in the toroidal Lie algebra  $\mathfrak{sp}_{2N}(\C_q)$ given in Proposition \ref{prop:dualpairs}.
Note that the embedding from  $\mathfrak{gl}_m(\C_q)$ (resp.\,$\mathfrak{sp}_{2m}(\C_q)$; resp.\,$\mathfrak{so}_m(\C_q)$)
 to $\mathfrak{sp}_{2N}(\C_q)$
can be lifted to an embedding from  $\widehat{\mathfrak{gl}}_m(\C_q)$ (resp.\,$\widehat{\mathfrak{sp}}_{2m}(\C_q)$; resp.\,$\widehat{\mathfrak{so}}_m(\C_q)$)
 to $\widehat{\mathfrak{sp}}_{2N}(\C_q)$ with $\bm{c}\mapsto n\bm{c}$.
 We obtain in this way the following three pairs of mutually commutative subalgebras in   $\widehat{\mathfrak{sp}}_{2N}(\C_q)$:
\begin{eqnarray} \label{eq:pairsinhatspwNcq}
 (\mathfrak{gl}_n,\widehat{\mathfrak{gl}}_m(\C_q)),\qquad\  (\mathfrak{so}_n, \widehat{\mathfrak{sp}}_{2m}(\C_q)),\qquad\
 (\mathfrak{sp}_{2n},\widehat{\mathfrak{so}}_m(\C_q)).
\end{eqnarray}

 By taking restriction to the first pair in \eqref{eq:pairsinhatspwNcq}, the oscillator $\widehat{\mathfrak{sp}}_{2N}(\C_q)$-module
  $\mathcal{F}_N(\bm{Z})$ becomes a $\mathfrak{gl}_n\oplus \widehat{\mathfrak{gl}}_m(\C_q)$-module.
As a $\mathfrak{gl}_n$-module,  $\mathcal{F}_N(\bm{Z})$ can be
  integrated to a $\mathrm{GL}_n$-module.
 When
 $\bm{Z}=\mathbb{Z}+\frac{1}{2}$, the resulting $\mathrm{GL}_n$-module  is locally regular.
 Thus, one obtains a canonical locally regular $(\mathrm{GL}_n,\widehat{\mathfrak{gl}}_m(\C_q))$-module structure on  $\mathcal{F}_N(\mathbb{Z}+\frac{1}{2})$.
For the case that $\bf{Z}=\mathbb{Z}$,  just like the classical case,
the $\mathrm{GL}_n$-module $\mathcal{F}_N(\bm{Z})$ is a direct sum of  irreducible  highest weight modules, which are not regular if $m$ is odd.
To make it a locally regular $\mathrm{GL}_n$-module when $m$ is odd, we  modify the action of $\mathfrak{gl}_n$ on $\mathcal{F}_N(\mathbb{Z})$ as follows:
\[E_{i,j}\mapsto \sum_{k=1}^m\sum_{r\in\mathbb{Z}}:\psi_k^i(-r)\bar{\psi}_k^j(r):-\delta_{i,j}\frac{m}{2}\]
for $1\le i,j\le n$, recalling the original action of $\mathfrak{gl}_n$ on $\mathcal{F}_N(\mathbb{Z})$ is given by
\[E_{i,j}\mapsto \sum_{k=1}^m\sum_{r\in\mathbb{Z}}:\psi_k^i(-r)\bar{\psi}_k^j(r):\]
for $1\le i,j\le n$.
With this modification, $\mathcal{F}_N(\mathbb{Z})$ is a $\mathfrak{gl}_n\oplus \widehat{\mathfrak{gl}}_m(\C_q)$-module as well, and
as a $\mathfrak{gl}_n$-module, it
can be integrated to a locally regular $\mathrm{GL}_n$-module.
Then we obtain in this way a locally regular $(\mathrm{GL}_n,\widehat{\mathfrak{gl}}_m(\C_q))$-module structure on  $\mathcal{F}_N(\mathbb{Z})$.

To be more precise, set
  \begin{eqnarray*}\label{eq:defUmn}
  U_m:=\mathbb{C}^{m}\otimes\mathbb{C}[t],\qquad U_m^n:=(\mathbb{C}^{n}\otimes U_m)\oplus ((\mathbb{C}^{n})^{\ast}\otimes U_m).
  \end{eqnarray*}
 View $U_m^n$ as a $\mathfrak{gl}_{n}$-module such that $\mathfrak{gl}_{n}$ acts naturally on $\mathbb{C}^{n}$ and its dual space $(\mathbb{C}^{n})^{\ast}$,
   and acts trivially on $U_m$.
   Then the symmetric algebra $\mathrm{S}(U_m^n)$ of $U_m^n$ is naturally a $\mathfrak{gl}_{n}$-module, and can be
   integrated to a locally regular $\mathrm{GL}_n$-module.
Let $\{e_1,e_2,\ldots,e_n\}$ be a basis   of $\mathbb{C}^{n}$, $\{e^1,e^2,\ldots,e^n\}$  its dual basis in $(\mathbb{C}^{n})^{\ast}$, and
   $\{w^1,w^2,\ldots,w^m\}$ a basis of  $\mathbb{C}^{m}$.
One can check that $\mathcal{F}_N(\bm{Z})$ is isomorphic to $\mathrm{S}(U_m^n)$ as $\mathfrak{gl}_{n}$-modules (and hence as $\mathrm{GL}_n$-modules).
The module isomorphism is given as follows
\begin{equation}\label{eq:glnmodiso}
\psi_i^k(r+2\epsilon)|0\rangle=e_{k}\otimes w^{i}\otimes t^{-r-\frac{1}{2}-\epsilon},\quad \bar{\psi}_i^k(r)|0\rangle=e^{k}\otimes w^{i}\otimes t^{-r-\frac{1}{2}-\epsilon}
\end{equation}
for $1\leq i\leq m$, $1\leq k\leq n$ and $r\in \bm{Z}$ with $r<0$, where $\epsilon$ is as in \eqref{eq:defepsilon}.

Similarly, by taking restriction to the second pair in \eqref{eq:pairsinhatspwNcq}, the oscillator $\widehat{\mathfrak{sp}}_{2N}(\C_q)$-module
  $\mathcal{F}_N(\bm{Z})$ becomes an $\mathfrak{so}_n\oplus \widehat{\mathfrak{sp}}_{2m}(\C_q)$-module.
Note that \eqref{eq:glnmodiso} also affords an $\mathfrak{so}_n$-module isomorphism from  $\mathcal{F}_N(\bm{Z})$ to  $\mathrm{S}(U_m^n)|_{\mathfrak{so}_n}$.
Thus, as an $\mathfrak{so}_n$-module module, $\mathcal{F}_N(\bm{Z})$ can be integrated to
an $\mathrm{SO}_n$-module and extended to a locally regular $\mathrm{O}_n$-module.
Then we obtain a locally regular $(\mathrm{O}_n,\widehat{\mathfrak{sp}}_{2m}(\C_q))$-module structure on $\mathcal{F}_N(\bm{Z})$.

Finally, by taking restriction to the third pair in \eqref{eq:pairsinhatspwNcq}, the oscillator $\widehat{\mathfrak{sp}}_{2N}(\C_q)$-module
  $\mathcal{F}_N(\bm{Z})$ becomes a $\mathfrak{sp}_{2n}\oplus \widehat{\mathfrak{so}}_{m}(\C_q)$-module.
  When $\bm{Z}=\mathbb{Z}$, similar to the classical case, as an $\mathfrak{sp}_{2n}$-module, $\mathcal{F}_N(\bm{Z})$ can not be integrated to an $\mathrm{Sp}_{2n}$-module.
However,  when $\bm Z=\mathbb{Z}+\frac{1}{2}$, the $\mathfrak{sp}_{2n}$-module $\mathcal{F}_N(\bm{Z})$ integrates to
a locally regular  $\mathrm{Sp}_{2n}$-module, and so is a locally regular $(\mathrm{Sp}_{2n}, \widehat{\mathfrak{so}}_{m}(\C_q))$-module.
In fact, set
\begin{eqnarray*}\label{eq:deftildeUmn}
\tilde{U}_m^n:=\mathbb{C}^{2n}\otimes U_m
\end{eqnarray*}
to be viewed as an $\mathfrak{sp}_{2n}$-module such that $\mathfrak{sp}_{2n}$ acts naturally on $\mathbb{C}^{2n}$ and acts trivially on $U_m$.
Then, as $\mathfrak{sp}_{2n}$-modules, $\mathcal{F}_N(\mathbb{Z}+\frac{1}{2})$ is isomorphic to
the symmetric algebra $\mathrm{S}(\tilde{U}_m^n)$.
The isomorphism is given by
\begin{eqnarray}\label{eq:sp2nmodiso}
\psi_i^k(r)|0\rangle=e_k\otimes w^{m+1-i}\otimes t^{-r-\frac{1}{2}},\quad \bar{\psi}_i^k(r)|0\rangle=-e_{n+k}\otimes w^i \otimes t^{-r-\frac{1}{2}}
\end{eqnarray}
for $1\leq i\leq m$, $1\leq k\leq n$ and $r\in\mathbb{Z}+\frac{1}{2}$ with $r<0$.

In summary,
let
$(\mathrm{G},\mathfrak{g})$ be one of the following pairs
\begin{eqnarray}\label{eq:pariGg}
 (\mathrm{GL}_n,\widehat{\mathfrak{gl}}_m(\C_q)),\qquad\  (\mathrm{O}_n, \widehat{\mathfrak{sp}}_{2m}(\C_q)),\qquad\
 (\mathrm{Sp}_{2n},\widehat{\mathfrak{so}}_m(\C_q)).
\end{eqnarray}
Then we have shown that,
except the case that $(\mathrm{G},\mathfrak{g})=(\mathrm{Sp}_{2n},\widehat{\mathfrak{so}}_m(\C_q))$
and $\mathbf{Z}=\mathbb{Z}$,   $\mathcal{F}_N(\mathbf{Z})$  is naturally a locally regular $(\mathrm{G},\mathfrak{g})$-module.

As the main result of this paper, we have the following generalization of the classical Howe dual pairs on $\mathcal{F}_N$, whose proof will be
given in next section.

\begin{thm}\label{thm:main} Let $N,n,m$ be positive integers such that $N=nm$, let $(\mathrm{G},\mathfrak{g})$ be one of the  pairs
 in \eqref{eq:pariGg}, and let $\bm{Z}=\Z$ or $\Z+\frac{1}{2}$. Then,
except the case that $(\mathrm{G},\mathfrak{g})=(\mathrm{Sp}_{2n},\widehat{\mathfrak{so}}_m(\C_q))$
and $\mathbf{Z}=\mathbb{Z}$, $(\mathrm{G},\mathfrak{g})$ forms a Howe dual pair on $\mathcal{F}_N(\bm{Z})$.
Furthermore, as a locally regular $(\mathrm{G},\mathfrak{g})$-module, we have the following  multiplicity-free decomposition:
\begin{eqnarray}\label{eq:maindec}
\mathcal{F}_N(\bm{Z})=\bigoplus_{\bm\mu\in \mathcal{R}(\mathrm{G})} L_{\mathrm{G}}(\bm\mu)
\otimes L_{\mathfrak{g},\bm{Z}}(\eta_{\bm\mu,\bm{Z}}).
\end{eqnarray}
\end{thm}

\section{Proof of Theorem \ref{thm:main}}
This section is devoted to the proof of Theorem \ref{thm:main}.
Throughout this section, let $N,m,n$, $(\mathrm{G},\mathfrak{g})$ and $\bm{Z}$ be as in Theorem \ref{thm:main}.

\subsection{Howe dual pairs on $\mathcal{F}_N(\bm{Z})$}
The main goal of this subsection is to prove the first assertion in Theorem \ref{thm:main}. Namely, we will prove the following result.

\begin{prpt}\label{prop:howedual}
$(\mathrm{G},\mathfrak{g})$ is a Howe dual pair on $\mathcal{F}_N(\bm Z)$.
\end{prpt}

Since $\mathcal{W}_N(\bm{Z})$ is stable under the conjugate
action of $\mathrm{G}$ (as subalgebras of $\mathrm{End}(\mathcal{F}_N(\bm Z))$),
 $\mathcal{W}_N(\bm{Z})$ is a $\mathrm{G}$-module with
\[g.x=g x g^{-1}
\]
for $g\in \mathrm{G}$ and $x\in \mathcal{W}_N(\bm{Z})$.
Denote by $\mathcal{W}_N(\bm Z)^{\mathrm{G}}$ the subalgebra of $\mathcal{W}_N(\bm Z)$ fixed by $\mathrm{G}$.

In what follows we are going to determine to the typical generators in  $\mathcal{W}_N(\bm Z)^{\mathrm{G}}$ by using the classical invariant theory. Set
\[ W_m:=\mathbb{C}^{m}\otimes \mathbb{C}[t,t^{-1}],\ W_m^n:=(\C^n\otimes W_m)\oplus ((\C^n)^*\otimes W_m),\
\tilde{W}_m^n:=\C^{2n}\otimes W_m.
\]
When $\mathrm{G}=\mathrm{GL}_n$ or $\mathrm{O}_n$, view $W_m^n$ as a $\mathrm{G}$-module such that $\mathrm{G}$ acts naturally on
$\C^n$ and its dual space $(\C^n)^*$, and acts trivially on $\C[t,t^{-1}]$.
Similar to \eqref{eq:glnmodiso}, one can check that  the $\mathrm{G}$-module $\mathcal{W}_N(\bm{Z})$ is isomorphic to the symmetric algebra
$\mathrm{S}(W_m^n)$,  where the isomorphism
is given by
\begin{equation}\label{eq:glnmodisow}
\psi_i^k(r)=e_{k}\otimes w^{i}\otimes t^{-r-\frac{1}{2}+\epsilon},\quad \bar{\psi}_i^k(r)=e^{k}\otimes w^{i}\otimes t^{-r-\frac{1}{2}-\epsilon}
\end{equation}
for $1\leq i\leq m$, $1\leq k\leq n$ and $r\in \bm{Z}$.

When $\mathrm{G}=\mathrm{Sp}_{2n}$ (and so $\bm{Z}=\Z+\frac{1}{2}$), view $\tilde{W}_m^n$ as a
$\mathrm{G}$-module such that $\mathrm{G}$ acts naturally on
$\C^{2n}$ and  acts trivially on $\C[t,t^{-1}]$.
Similar to \eqref{eq:sp2nmodiso},
the $\mathrm{G}$-module $\mathcal{W}_N(\bm{Z})$ is isomorphic to $\mathrm{S}(\tilde{W}_m^n)$, where the isomorphism
is given by
\begin{equation}\label{eq:sp2nmodisow}
\psi_i^k(r)=e_k\otimes w^{m+1-i}\otimes t^{r-\frac{1}{2}},\quad \bar{\psi}_i^k(r)=-e_{n+k}\otimes w^{i}\otimes t^{r-\frac{1}{2}}
\end{equation}
for $1\leq i\leq m$, $1\leq k\leq n$ and $r\in \Z+\frac{1}{2}$.

In view of the modules isomorphisms \eqref{eq:glnmodisow} and \eqref{eq:sp2nmodisow}, one can conclude from the classical invariant theory (\cite{GW}) that
\begin{itemize}
\item  $\mathcal{W}_N(\bm Z)^{\mathrm{GL}_n}$ is generated by the  elements
\begin{equation}\label{eq:invgl}
\sum_{k=1}^{n}\psi_i^k(a)\bar{\psi}_j^k(b),\qquad 1\leq i,j\leq m,\ a,b\in\bm Z.
\end{equation}
\item $\mathcal{W}_N(\bm Z)^{\mathrm{O}_n}$ is generated by the elements
  \begin{equation}\label{eq:invo}
\sum_{k=1}^{n}\psi_i^k(a)\bar{\psi}_j^k(b),\quad \sum_{k=1}^{n}\psi_i^k(a)\psi_j^{n+1-k}(b),\quad \sum_{k=1}^{n}\bar{\psi}_i^k(a)\bar{\psi}_j^{n+1-k}(b),
\quad 1\leq i,j\leq m,\ a,b\in\bm Z.
 \end{equation}
\item $\mathcal{W}_N(\bm Z)^{\mathrm{Sp}_{2n}}$ (with $\bm{Z}=\Z+\frac{1}{2}$) is generated by the  elements
\begin{equation}\label{eq:invsp}
  \sum_{k=1}^{n}(\psi_i^k(a)\bar{\psi}_j^k(b)-\bar{\psi}_{m+1-i}^k(a)\psi_{m+1-j}^k(b)),\quad 1\leq i,j\leq m,\ a,b\in\bm Z.
  \end{equation}
\end{itemize}

On the other hand, by using the embeddings \eqref{eq:dualpair1}-\eqref{eq:dualpair3} and the module action \eqref{ho}, we have
\begin{itemize}
\item the $\widehat{\mathfrak{gl}}_m(\C_q))$-module action on $\mathcal{F}_N(\bm{Z})$ is given by $\bm{c}=-n$ and
\begin{equation}\label{eq:glmmodact}
E_{i,j}t_0^at_1^b=\sum_{k=1}^{n}\sum_{r\in\bm Z}q^{-br}:\psi_i^k(a-r)\bar{\psi}_j^k(r):-n\delta_{a,0}\delta_{i,j}\omega_{\bm Z}(b)
\end{equation}
for $1\le i,j\le m$ and $a,b\in \Z$.

\item the $\widehat{\mathfrak{sp}}_{2m}(\C_q))$-module action on $\mathcal{F}_N(\bm{Z})$ is given by $\bm{c}=-n$ and
\begin{equation}\begin{split}\label{eq:sp2mmodact}
f_{i,j}(a,b)&=\sum_{k=1}^{n}\sum_{r\in\bm Z}q^{-br}:\psi_i^k(a-r)\bar{\psi}_j^k(r):-n\delta_{a,0}\delta_{i,j}\omega_{\bm Z}(b),\\
 g_{i,j}(a,b)&=\sum_{k=1}^{n}\sum_{r\in\bm Z}q^{-br}:\psi_i^k(a-r)\psi_j^{n+1-k}(r):,\\
h_{i,j}(a,b)&=\sum_{k=1}^{n}\sum_{r\in\bm Z}q^{-br}:\bar{\psi}_i^k(a-r)\bar{\psi}_j^{n+1-k}(r):
\end{split}
\end{equation}
for $1\le i,j\le m$ and $a,b\in \Z$.

\item the $\widehat{\mathfrak{so}}_m(\C_q))$-module action on $\mathcal{F}_N(\bm{Z})$ (with $\bm{Z}=\Z+\frac{1}{2}$) is given by $\bm{c}=-n$ and
\begin{equation}\label{eq:sommodact}
e_{i,j}(a,b)=\sum_{k=1}^{n}\sum_{r\in\bm Z}q^{-br}(:\psi_i^k(a-r)\bar{\psi}_j^k(r):-:\psi_{m+1-j}^k(r)\bar{\psi}_{m+1-i}^k(a-r):)
-2n\delta_{a,0}\delta_{i,j}\omega_{\bm Z}(b)
\end{equation}
for $1\le i,j\le m$ and $a,b\in \Z$.
\end{itemize}

 Then we have the following result.

 \begin{lemt}\label{le:howedual}
 Let $U$ be a subspace of $\mathcal{F}_N(\bm Z)$. Then $U$ is a $\mathfrak{g}$-submodule if and only if $U$ is a $\mathcal{W}_N(\bm Z)^{\mathrm{G}}$-submodule. Furthermore, two $\mathfrak{g}$-submodules of $\mathcal{F}_N(\bm Z)$ are isomorphic if and only if they are isomorphic as $\mathcal{W}_N(\bm Z)^{\mathrm{G}}$-submodules.
 \end{lemt}
 \begin{proof}
When $(\mathrm{G},\mathfrak{g})=(\mathrm{GL}_n,\widehat{\mathfrak{gl}}_m(\C_q))$ (resp.\,$(\mathrm{O}_n,\widehat{\mathfrak{sp}}_{2m}(\C_q))$;
resp.\,$(\mathrm{Sp}_{2n},\widehat{\mathfrak{so}}_m(\C_q))$),
by comparing \eqref{eq:invgl} (resp.\,\eqref{eq:invo}; resp.\,\eqref{eq:invsp}) with
\eqref{eq:glmmodact} (resp.\,\eqref{eq:sp2mmodact}; resp.\,\eqref{eq:sommodact}), it follows that every
$\mathcal{W}_N(\bm Z)^{\mathrm{G}}$-submodule of $\mathcal{F}_N(\bm Z)$  is also a $\mathfrak{g}$-submodule.
Furthermore, every module homomorphism between two $\mathcal{W}_N(\bm Z)^{\mathrm{G}}$-submodules of $\mathcal{F}_N(\bm Z)$
is also a $\mathfrak{g}$-module homomorphism.

Conversely, let $U,U'$ be two  $\mathfrak{g}$-submodules of $\mathcal{F}_N(\bm{Z})$, and let $f:U\rightarrow U'$ be a $\mathfrak{g}$-module
homomorphism.
Let $1\leq i,j\leq m,~a\in \mathbb{Z}$ and $u\in U$.
When $(\mathrm{G},\mathfrak{g})=(\mathrm{GL}_n,\widehat{\mathfrak{gl}}_m(\C_q))$, pick a  finite subset $T$ of $\bm{Z}$ such that
\[\sum_{k=1}^{n}:\psi_i^k(a-r)\bar{\psi}_j^k(r):u=0
\]
for all  $r\notin T$. Consider the equations
\begin{eqnarray*}
&&(E_{i,j}t_0^at_1^b-\delta_{a,0}\delta_{i,j}\omega_{\bm Z}(b)\bm{c}).u=\sum_{k=1}^{n}\sum_{r\in\bm Z}q^{-rb}:\psi_i^k(a-r)\bar{\psi}_j^k(r):u\\
&=& \sum_{r\in T} q^{-rb} (\sum_{k=1}^{n}:\psi_i^k(a-r)\bar{\psi}_j^k(r):u) \in U\end{eqnarray*}
  for $1\leq b\leq \mathrm{Card}\,T$.
   As $q$ is not a root of unity, the coefficient matrix of the above equations is a Vandermonde matrix.
 By solving the equations, we get
  $$\sum_{k=1}^{n}:\psi_i^k(a-r)\bar{\psi}_j^k(r):u\in U$$
  for $1\leq i,j\leq m$, $a\in\mathbb{Z}$ and $r\in\bm Z$.
  This together with \eqref{eq:invgl} gives that $U$ is $\mathcal{W}_N(\bm Z)^{\mathrm{GL}_n}$-submodule of $\mathcal{F}_N(\bm Z)$.
  Furthermore, by using the fact that
  \[f((E_{i,j}t_0^at_1^b-\delta_{a,0}\delta_{i,j}\omega_{\bm Z}(b)\bm{c}).u)
  =(E_{i,j}t_0^at_1^b+\delta_{a,0}\delta_{i,j}\omega_{\bm Z}(b)\bm{c}).f(u),
  \]
  we have
  \[
  \sum_{r\in T} q^{-rb} f(\sum_{k=1}^{n}:\psi_i^k(a-r)\bar{\psi}_j^k(r):u)=
   \sum_{r\in T} q^{-rb} (\sum_{k=1}^{n}:\psi_i^k(a-r)\bar{\psi}_j^k(r):f(u))
  \]
   for $1\leq b\leq \mathrm{Card}\,T$.
Thus, by solving the above equations, we have
\[f(\sum_{k=1}^{n}:\psi_i^k(a-r)\bar{\psi}_j^k(r):u)=\sum_{k=1}^{n}:\psi_i^k(a-r)\bar{\psi}_j^k(r):f(u)\]
for $1\leq i,j\leq m$, $a\in\mathbb{Z}$ and $r\in\bm Z$, which together with  \eqref{eq:invgl} gives that
$f$ is a $\mathcal{W}_N(\bm Z)^{\mathrm{GL}_n}$-module homomorphism as well.

Similarly, when $(\mathrm{G},\mathfrak{g})=(\mathrm{O}_n,\widehat{\mathfrak{sp}}_{2m}(\C_q))$, by solving the equations
\begin{eqnarray*}(f_{i,j}(a,b)-\delta_{a,0}\delta_{i,j}\omega_{\bm Z}(b)\bm{c}).u &=&\sum_{k=1}^{n}\sum_{r\in\bm Z}q^{-br}:\psi_i^k(a-r)\bar{\psi}_j^k(r):u\in U,\\
g_{i,j}(a,b).u&=&\sum_{k=1}^{n}\sum_{r\in\bm Z}q^{-br}:\psi_i^k(a-r)\psi_j^{n+1-k}(r):u\in U,\\
h_{i,j}(a,b).u&=&\sum_{k=1}^{n}\sum_{r\in\bm Z}q^{-br}:\bar{\psi}_i^k(a-r)\bar{\psi}_j^{n+1-k}(r):u\in U,
\end{eqnarray*}
we find that
\[:\psi_i^k(a-r)\bar{\psi}_j^k(r):u,\ :\psi_i^k(a-r)\psi_j^{n+1-k}(r):u,\ :\bar{\psi}_i^k(a-r)\bar{\psi}_j^{n+1-k}(r):u\in U\]
for $1\leq i,j\leq m$, $a\in\mathbb{Z}$ and $r\in\bm Z$.
  This together with \eqref{eq:invo} gives that $U$ is $\mathcal{W}_N(\bm Z)^{\mathrm{O}_n}$-submodule of $\mathcal{F}_N(\bm Z)$.
One can also prove in this way that $f$ is a $\mathcal{W}_N(\bm Z)^{\mathrm{O}_n}$-module homomorphism.

  When $(\mathrm{G},\mathfrak{g})=(\mathrm{Sp}_{2n},\widehat{\mathfrak{so}}_m(\C_q)))$, one can  deduce from
\eqref{eq:sommodact} that
\[\sum_{k=1}^{n}(:\psi_i^k(a-r)\bar{\psi}_j^k(r):-:\psi_{m+1-j}^k(r)\bar{\psi}_{m+1-i}^k(a-r):)u\in U\]
  for $1\leq i,j\leq m$, $a\in\mathbb{Z}$ and $r\in \Z+\frac{1}{2}$.
  Thus, it follows from \eqref{eq:invsp} that $U$ is $\mathcal{W}_N(\Z+\frac{1}{2})^{\mathrm{Sp}_{2n}}$-submodule of
   $\mathcal{F}_N(\Z+\frac{1}{2})$. Similarly, one can check that $f$ is a $\mathcal{W}_N(\bm Z)^{\mathrm{Sp}_{2n}}$-module homomorphism.
\end{proof}

\textbf{Proof of Proposition \ref{prop:howedual}}:
In view of \cite[Theorem 4.2.1]{GW}, $(\mathrm{G}, \mathcal{W}_N(\bm Z)^{\mathrm{G}})$ is a Howe dual pair on $\mathcal{F}_N(\bm Z)$.
It then follows from Lemma $\ref{le:howedual}$ that  $(\mathrm{G},\mathfrak{g})$ is a Howe dual pair on $\mathcal{F}_N(\bm Z)$ as well,
which proves Proposition \ref{prop:howedual}.

\subsection{$(\mathrm{GL}_n,\widehat{\mathfrak{gl}}_m(\mathbb{C}_q))$ and $(\mathrm{Sp}_{2n},\widehat{\mathfrak{so}}_{m}(\C_q))$-highest
 weight vectors in $\mathcal{F}_N(\bm{Z})$}
The rest of this section is devoted to a proof of  the multiplicity-free decomposition \eqref{eq:maindec}.
Set  $\mathrm{G}'=\mathrm{G}$ if
$\mathrm{G}=\mathrm{GL}_n$ or $\mathrm{Sp}_{2n}$, and $\mathrm{G}'=\mathrm{SO}_n$ if $\mathrm{G}=\mathrm{O}_n$.
Then the proof of \eqref{eq:maindec}
will be reduced to determine the $(\mathrm{G}',\mathfrak{g})$-highest weight vectors in
$\mathcal{F}_N(\bm{Z})$ in the following sense:

\begin{dfnt}
Let $0\ne v\in \mathcal{F}_N(\bm{Z})$ and $(\bm\mu,\eta)\in \mathcal{R}(\mathrm{G}')\times \mathfrak{g}_0^*$.
 We say that $v$ is a $(\mathrm{G}',\mathfrak{g})$-highest weight vector in $\mathcal{F}_N(\bm{Z})$ with
highest weight $(\bm\mu,\eta)$ if
\[ \mathrm{N}^+_{\mathrm{G}'}.v=v, \quad\mathfrak{g}_{+,\bm{Z}}.v=0,\quad h.v=h^{\bm\mu}v,\quad x.v=\eta(x)v
\]
for all $h\in \mathrm{H}_{\mathrm{G}'}$  and $x\in \mathfrak{g}_0$.
\end{dfnt}

In this subsection, we determine the $(\mathrm{GL}_n,\widehat{\mathfrak{gl}}_m(\mathbb{C}_q))$
and $(\mathrm{Sp}_{2n},\widehat{\mathfrak{so}}_{m}(\C_q))$-highest weight vectors in $\mathcal{F}_N(\bm{Z})$.
For convenience, we define $\phi_r^k,\bar{\phi_r}^k$, $1\leq k\leq n,~r\in\bm Z$ by
$$\phi_{(r+\frac{1}{2}-\epsilon)m-i+\frac{1}{2}-\epsilon}^k=\psi_i^k(r),\quad \bar{\phi}_{(r-\frac{1}{2}+\epsilon)m+i-\frac{1}{2}-\epsilon}^k=\bar{\psi}_i^k(r)$$
for $1\leq i\leq m,~r\in\bm Z$. Then for $1\leq k,\ell\leq n,~a,b\in\bm Z$ we have
\begin{eqnarray}\label{eq:commu}
\label{gla}\phi_a^k\phi_b^\ell-\phi_b^\ell\phi_a^k=0,\quad \bar{\phi}_a^k\bar{\phi}_b^\ell-\bar{\phi}_b^\ell\bar{\phi}_a^k=0,\quad
\bar{\phi}_a^k\phi_b^\ell-\phi_b^\ell\bar{\phi}_a^k=\delta_{k,\ell}\delta_{a+b,-2\epsilon}.
\end{eqnarray}

For convenience, for $a\in\mathbb{Z}$, $1\leq i\leq m$, $r\in\bm Z$, we set $\alpha(r,a,i)=(a-r+\frac{1}{2}-\epsilon)m-i+1$, $\beta(r,a,i)=(r+a-\frac{1}{2}+\epsilon)m+j$.
 And, for a positive integer $b$, we set $I_b=\{1,2,\ldots,b\}$.

Let $\ell=1,2,\dots,n$. We define two elements $A_{\ell,{\bm Z}}$ and $\bar{A}_{\ell,\bm Z}$ of the Weyl algebra $\mathcal{W}_N(\bm Z)$ in terms of $\phi_a^k$ and $\bar{\phi}_a^k$ by
\begin{equation}\label{v1}
 A_{\ell,\bm Z}:=\textrm{det}\ \begin{pmatrix}
   \phi_{-\frac{1}{2}-\epsilon}^1 & \phi_{-\frac{3}{2}-\epsilon}^1 & \cdots & \phi_{-\ell+\frac{1}{2}-\epsilon}^1 \\
   \phi_{-\frac{1}{2}-\epsilon}^2 & \phi_{-\frac{3}{2}-\epsilon}^2 & \cdots & \phi_{-\ell+\frac{1}{2}-\epsilon}^2 \\
   \cdots& \cdots & \ddots & \cdots \\
   \phi_{-\frac{1}{2}-\epsilon}^\ell & \phi_{-\frac{3}{2}-\epsilon}^\ell & \cdots & \phi_{-\ell+\frac{1}{2}-\epsilon}^\ell
 \end{pmatrix}
\end{equation}
and
\begin{equation}\label{v2}
 \bar{A}_{\ell,\bm Z}:=\textrm{det}\ \begin{pmatrix}
   \bar{\phi}_{-\frac{1}{2}-\epsilon}^n & \bar{\phi}_{-\frac{3}{2}-\epsilon}^n & \cdots & \bar{\phi}_{\ell-n-\frac{1}{2}-\epsilon}^n \\
   \bar{\phi}_{-\frac{1}{2}-\epsilon}^{n-1} & \bar{\phi}_{-\frac{3}{2}-\epsilon}^{n-1} & \cdots & \bar{\phi}_{\ell-n-\frac{1}{2}-\epsilon}^{n-1} \\
   \cdots& \cdots & \ddots & \cdots \\
   \bar{\phi}_{-\frac{1}{2}-\epsilon}^{\ell} & \bar{\phi}_{-\frac{3}{2}-\epsilon}^{\ell} & \cdots & \bar{\phi}_{\ell-n-\frac{1}{2}-\epsilon}^{\ell}
 \end{pmatrix}.
\end{equation}

We start with the following result.

\begin{lemt}\label{le:dualpair1le1}
Let $1\leq i,j\leq m$, $a\in\mathbb{Z}$, $r\in \bm{Z}$ and $\ell=1,2,\dots,n$.
Assume that either $a>0$ or $a=0$ and $i<j$. Then
\begin{equation}\label{eq:dualpair1eq1}
[\sum_{k=1}^{n} \psi_i^k(a-r)\bar{\psi}_j^k(r),A_{\ell,\bm Z}]=
\begin{cases}
A_{\ell,\bm Z}^{i,j,r}&\mbox{if } \beta(r,0,j)\in I_\ell\\
0&\text{otherwise}
\end{cases},
\end{equation}
where $A_{\ell,\bm Z}^{i,j,r}$ is the determinant obtained  by replacing the $\beta(r,0,j)$-column in $A_{\ell,\bm Z}$ by the column vector
      $$(\phi_{\alpha(r,a,i)-\frac{1}{2}-\epsilon}^1,\phi_{\alpha(r,a,i)-\frac{1}{2}-\epsilon}^2,\dots,
      \phi_{\alpha(r,a,i)-\frac{1}{2}-\epsilon}^\ell)^t. $$
Furthermore, for the case that $\beta(r,0,j)\in I_\ell$, we have
\begin{eqnarray}\label{eq:dualpair1eq2}
[A_{\ell',\bm Z},A_{\ell,\bm Z}^{i,j,r}]=[\bar{A}_{\ell'',\bm Z},A_{\ell,\bm Z}^{i,j,r}]=0\quad\text{and}\quad
A_{\ell,\bm Z}^{i,j,r}|0\rangle=0,
\end{eqnarray}
where $\ell',\ell''=1,2,\dots,n$ with $\ell'\ge \ell$ and $\ell''>\ell$.
\end{lemt}
\begin{proof}
The verification of  \eqref{eq:dualpair1eq1} is straightforward and is omitted.
For the equation \eqref{eq:dualpair1eq2}, we have from \eqref{eq:commu} that $[A_{\ell',\bm Z},A_{\ell,\bm Z}^{i,j,r}]=0$ and
 $[\bar{A}_{\ell'',\bm Z},A_{\ell,\bm Z}^{i,j,r}]=0$
 since $\ell''>\ell$.
     If $\beta(r,0,j)\in I_\ell$,
     then we have
     $$\alpha(r,a,i)-\frac{1}{2}-\epsilon=-\beta(r,0,j)+\frac{1}{2}+am+j-i-\epsilon\geq -\ell+\frac{3}{2}-\epsilon,$$
 which implies $A_{\ell,\bm Z}^{i,j,r}|0\rangle=0$.
 \end{proof}

 Similar to Lemma \ref{le:dualpair1le1}, we have the following result.

\begin{lemt}\label{le:dualpair1le1bar}
Let $1\leq i,j\leq m$, $a\in\mathbb{Z}$, $r\in \bm{Z}$ and $\ell=1,2,\dots,n$.
Assume that either $a>0$ or $a=0$ and $i<j$. Then
\begin{equation}
[\sum_{k=1}^{n} \psi_i^k(a-r)\bar{\psi}_j^k(r),\bar{A}_{\ell,\bm Z}]=
\begin{cases}
\bar{A}_{\ell,\bm Z}^{i,j,r}&\mbox{if }\alpha(r,a,i)\in I_{n-\ell+1}\\
0&\mbox{otherwise}
\end{cases},
\end{equation}
 where
 $\bar{A}_{\ell,\bm Z}^{i,j,r}$ is the determinant obtained by replacing the $\alpha(r,a,i)$-column in  $\bar{A}_{\ell,\bm Z}$ by the
column vector
 $$(-\bar{\phi}_{\beta(r,0,j)-\frac{1}{2}-\epsilon}^n,-\bar{\phi}_{\beta(r,0,j)-\frac{1}{2}-\epsilon}^{n-1},
 \dots,-\bar{\phi}_{\beta(r,0,j)-\frac{1}{2}-\epsilon}^{\ell})^t.$$
 Furthermore, for the case that $\alpha(r,a,i)\in I_{n-\ell+1}$, we have
 \begin{eqnarray}
[\bar{A}_{\ell',\bm Z},\bar{A}_{\ell,\bm Z}^{i,j,r}]=0\quad\text{and}\quad
\bar{A}_{\ell,\bm Z}^{i,j,r}|0\rangle=0,
\end{eqnarray}
where $\ell'=1,2,\dots,n$ with $\ell'\ge \ell$.
\end{lemt}

The following result can be  checked directly.

\begin{lemt}\label{le:dualpair1le2}
Let $1\leq p\leq n-1$ and $1\le \ell\le n$. Then we have
\begin{equation*}
  [\sum_{k=1}^{m}\sum_{r\in\bm Z}\psi_k^p(-r)\bar{\psi}_k^{p+1}(r),A_{\ell,\bm Z}]=
  [\sum_{k=1}^{m}\sum_{r\in\bm Z}\psi_k^p(-r)\bar{\psi}_k^{p+1}(r),\bar{A}_{\ell,\bm Z}]=0.
\end{equation*}
\end{lemt}

For $\ell=1,2,\dots,n$, define two elements $\bm{\mu}_\ell$ and $\bar{\bm{\mu}}_\ell$ in $\mathcal{R}(\mathrm{GL}_n)$ by
\[\bm{\mu}_\ell:=(\underbrace{1,1,\ldots,1}_\ell,0,0,\ldots,0)\quad\text{and}\quad
\bar{\bm{\mu}}_\ell:=(\underbrace{0,0,\ldots,0}_{\ell-1},-1,-1,\ldots,-1).\]

\begin{lemt}\label{le:dualpair1le3}
For $\ell=1,2,\dots,n$ and $h\in \mathrm{H}_{\mathrm{GL}_n}$, we have
\begin{equation*}
  h.A_{\ell,\bm Z}|0\rangle=h^{\bm{\mu}_\ell} A_{\ell,\bm Z}|0\rangle
  \quad\text{and}\quad h.\bar{A}_{\ell,\bm Z}|0\rangle=h^{\bar{\bm{\mu}}_\ell} \bar{A}_{\ell,\bm Z}|0\rangle.
\end{equation*}
\end{lemt}
\begin{proof}
For $1\leq p \leq n$, as operator on $\mathcal{F}_N(\bm{Z})$, we have
     $$E_{p,p}=\sum_{k=1}^{m}\sum_{r\in\bm Z}:\psi_k^p(-r)\bar{\psi}_k^p(r):-\epsilon m=\sum_{r\ge \frac{1}{2}-\epsilon}\phi_{-r}^p\bar{\phi}_{r}^p+ \sum_{r\leq -\frac{1}{2}-\epsilon}\bar{\phi}_{r}^p\phi_{-r}^p.$$
Using this, it follows that
   \begin{equation*}
   E_{p,p}A_{\ell,\bm Z}=
     \begin{cases}
       A_{\ell,\bm Z}E_{p,p}+A_{\ell,\bm Z} & \mbox{if } p \le \ell \\
       0, & \mbox{otherwise}
     \end{cases},
   \end{equation*}
   and
     \begin{equation*}
   E_{p,p}\bar{A}_{\ell,\bm Z}=
     \begin{cases}
       A_{\ell,\bm Z}E_{p,p}-\bar{A}_{\ell,\bm Z} & \mbox{if }  p\ge \ell \\
       0 & \mbox{otherwise}
     \end{cases}.
   \end{equation*}
 Thus, we find that
    $$h.A_{\ell,\bm Z}|0\rangle=h^{\bm{\mu}} A_{\ell,\bm Z}|0\rangle\quad\text{and}\quad h.\bar{A}_{\ell,\bm Z}|0\rangle=h^{\bar{\bm{\mu}}}\bar{A}_{\ell,\bm Z}|0\rangle$$
   for $h\in \mathrm{H}_{\mathrm{GL}_n}$.
\end{proof}

\begin{lemt}\label{le:dualpair1le4}
For $\ell=1,2,\dots,n$ and $x\in \widehat{\mathfrak{gl}}_m(\mathbb{C}_q)_0$, we have
\begin{equation*}
x.A_{\ell,\bm Z}|0\rangle=\eta_{\bm{\mu}_\ell,\bm{Z}}(x) A_{\ell,\bm Z}|0\rangle\quad \text{and}\quad x.\bar{A}_{\ell,\bm Z}|0\rangle=
\eta_{\bar{\bm{\mu}_\ell},\bm{Z}}(x) \bar{A}_{\ell,\bm Z}|0\rangle.
\end{equation*}
\end{lemt}
\begin{proof}
Recall that $\widehat{\mathfrak{gl}}_m(\mathbb{C}_q)_0$ is spanned by the elements $\bf{c}$ and $E_{i,i}t_1^b$ for $1\le i\le m$ and $b\in \Z$.
And, as operator on $\mathcal{F}_N(\bm{Z})$, we have
    $$E_{i,i}t_1^b=\sum_{r\in\bm Z}q^{-br}\left(\sum_{k=1}^n:\psi_i^k(-r)\bar{\psi}_i^k(r):\right)-n\omega_{\bm Z}(b).$$
Then the assertion  can be deduced directly by using Lemma \ref{le:dualpair1le1} and Lemma \ref{le:dualpair1le1bar}.
\end{proof}

For every $\bm{\mu}=(\mu_1,\mu_2,\dots,\mu_n)\in\mathcal{R}(\mathrm{GL}_n)$, we set
  $$v_{\bm{\mu},\bm Z}:=A_{1,\bm Z}^{\mu_1-\mu_2}A_{2,\bm Z}^{\mu_2-\mu_3}\ldots A_{p-1,\bm Z}^{\mu_{p-1}-\mu_p}A_{p,\bm Z}^{\mu_p}\bar{A}_{s,\bm Z}^{-\mu_{s}}\bar{A}_{s+1,\bm Z}^{\mu_{s}-\mu_{s+1}}\cdots \bar{A}_{n,\bm Z}^{\mu_{n-1}-\mu_{n}}|0\rangle\in \mathcal{F}_N(\bm{Z}),$$
  where the positive numbers $p,s$ are  as in  \eqref{eq:muformGLn}. We have:

\begin{prpt}\label{pr:dualpair1}
For every $\bm{\mu}\in\mathcal{R}(\mathrm{GL}_n)$, $v_{\bm{\mu},\bm Z}$ is a
  $(\mathrm{GL}_n,\widehat{\mathfrak{gl}}_m(\mathbb{C}_q))$-highest weight vector in $\mathcal{F}_N(\bm{Z})$
  with highest weight $(\bm\mu,\eta_{\bm\mu,\bm{Z}})$.
\end{prpt}
\begin{proof}
Recall that the action of $\widehat{\mathfrak{gl}}_m(\mathbb{C}_q)_{+,\bm{Z}}$  on $\mathcal{F}_N(\bm{Z})$ is as follows:
\[
E_{i,j}t_0^at_1^b=\sum_{r\in\bm Z}q^{-br}\left(\sum_{k=1}^n:\psi_i^k(a-r)\bar{\psi}_i^k(r):\right),
\]
where $1\le i,j\le m$ and $a,b\in \Z$ such that either $a>0$ or $a=0$ and $i<j$.
Then one can conclude from  Lemma \ref{le:dualpair1le1} and Lemma \ref{le:dualpair1le1bar} that
\begin{equation*} \widehat{\mathfrak{gl}}_m(\mathbb{C}_q)_{+,\bm{Z}}.v_{\bm{\mu},\bm{Z}}=0.
\end{equation*}
Similarly, Lemma \ref{le:dualpair1le2} implies that
\begin{equation*}
\mathrm{N}^+_{\mathrm{GL}_n}.v_{\bm{\mu},\bm{Z}}=v_{\bm{\mu},\bm{Z}},
\end{equation*}
while Lemma \ref{le:dualpair1le3} and Lemma \ref{le:dualpair1le4} imply that
\[
x.v_{\bm{\mu},\bm{Z}}=\eta_{\bm{\mu},\bm{Z}}(x) v_{\bm{\mu},\bm{Z}}\quad\text{and}\quad h.v_{\bm{\mu},\bm{Z}}
=h^{\bm\mu}v_{\bm{\mu},\bm{Z}}
\]
for $x\in \widehat{\mathfrak{gl}}_m(\mathbb{C}_q)_0$ and $h\in \mathrm{H}_{\mathrm{GL}_n}$.
This completes the proof.
\end{proof}

For every $\bm{\mu}=(\mu_1,\mu_2,\dots,\mu_{2n})\in \mathcal{R}(\mathrm{Sp}_{2n})$, set
$$v_{\bm{\mu},\bm Z}:=A_{1,\bm Z}^{\mu_1-\mu_2}A_{2,\bm Z}^{\mu_2-\mu_3}\cdots A_{n,\bm Z}^{\mu_{n}-\mu_{n+1}}|0\rangle\in \mathcal{F}_N(\bm{Z}).$$
Similar to Proposition \ref{pr:dualpair1}, we have

\begin{prpt}\label{pr:dualpair2}
When $\bm Z=\mathbb{Z}+\frac{1}{2}$, for every  $\bm\mu\in \mathcal{R}(\mathrm{Sp}_{2n})$,
$v_{\bm{\mu},\bm Z}$ is a
  $(\mathrm{Sp}_{2n},\widehat{\mathfrak{so}}_m(\mathbb{C}_q))$-highest weight vector in $\mathcal{F}_N(\bm{Z})$
  with highest weight $(\bm\mu,\eta_{\bm\mu,\bm{Z}})$.
\end{prpt}
\begin{proof}
Firstly, recall that
\[\widehat{\mathfrak{so}}_m(\mathbb{C}_q)_{+,\bm{Z}}=\widehat{\mathfrak{so}}_m(\mathbb{C}_q)\cap \widehat{\mathfrak{gl}}_m(\mathbb{C}_q)_{+,\bm{Z}},\quad
\widehat{\mathfrak{so}}_m(\mathbb{C}_q)_{0}=\widehat{\mathfrak{so}}_m(\mathbb{C}_q)\cap \widehat{\mathfrak{gl}}_m(\mathbb{C}_q)_{0}.\]
Then it follows from Proposition \ref{pr:dualpair1} that
\[\widehat{\mathfrak{so}}_m(\mathbb{C}_q)_{+,\bm{Z}}.v_{\bm{\mu},\bm Z}=0,\qquad x.v_{\bm{\mu},\bm Z}=\eta_{\bm\mu,\bm{Z}}(x)v_{\bm{\mu},\bm Z}\]
for $x\in \widehat{\mathfrak{so}}_m(\mathbb{C}_q)_{0}$.
Secondly, Lemma \ref{le:dualpair1le2} implies that the elements
\[f_{p,s}=\sum_{k=1}^m\sum_{r\in\frac{1}{2}+\mathbb{Z}}:\psi_k^p(-r)\bar{\psi}_k^s(r):,\quad 1\leq p< s\leq n\]
act trivially on $v_{\bm{\mu},\bm Z}$, and it is routine to check that the elements
\[g_{p,s}=\sum_{k=1}^m\sum_{r\in\frac{1}{2}+\mathbb{Z}}:\psi_k^p(-r)\psi_{m+1-k}^s(r):,\quad 1\leq p\leq s\leq n\]
act trivially on $v_{\bm{\mu},\bm Z}$ as well. This means  that
\[\mathrm{N}_{\mathrm{Sp}_{2n}}^+.v_{\bm{\mu},\bm Z}=v_{\bm{\mu},\bm Z}.\]
 Finally, by using Lemma \ref{le:dualpair1le3}, we see that
 \[
 h.v_{\bm{\mu},\bm Z}=h^{\bm{\mu}} v_{\bm{\mu},\bm Z}
 \]
 for $h\in \mathrm{H}_{\mathrm{Sp}_{2n}}$. This completes the proof.
\end{proof}

\subsection{$(\mathrm{SO}_n,\widehat{\mathfrak{sp}}_{2m}(\mathbb{C}_q))$-highest weight vectors in $\mathcal{F}_N(\bm{Z})$}
In this subsection, we determine the $(\mathrm{SO}_n,\widehat{\mathfrak{sp}}_{2m}(\mathbb{C}_q))$-highest weight vectors in $\mathcal{F}_N(\bm{Z})$
and complete the proof of Theorem \ref{thm:main}.

For every $\ell=1,2,\dots,n$, we define an element $B_{\ell,\bm Z}$ in the Weyl algebra $\mathcal{W}_N(\bm Z)$ by
\begin{equation}
B_{\ell,\bm Z}=\textrm{det}
\begin{pmatrix}
 \phi_{-\frac{1}{2}-\epsilon}^1 & \phi_{-\frac{3}{2}-\epsilon}^1 & \cdots & \phi_{-\bar{\ell}_{1}+\frac{1}{2}-\epsilon}^1&\bar{\phi}_{-\frac{1}{2}-\epsilon}^n & \bar{\phi}_{-\frac{3}{2}-\epsilon}^n & \cdots & \bar{\phi}_{-\bar{\ell}_{2}+\frac{1}{2}-\epsilon}^n \\
  \phi_{-\frac{1}{2}-\epsilon}^2 & \phi_{-\frac{3}{2}-\epsilon}^2 & \cdots & \phi_{-\bar{\ell}_{1}+\frac{1}{2}-\epsilon}^2&\bar{\phi}_{-\frac{1}{2}-\epsilon}^{n-1} & \bar{\phi}_{-\frac{3}{2}-\epsilon}^{n-1} & \cdots & \bar{\phi}_{-\bar{\ell}_{2}+\frac{1}{2}-\epsilon}^{n-1} \\
  \cdots&\cdots&\ddots&\cdots&\cdots&\cdots&\ddots&\cdots\\
  \phi_{-\frac{1}{2}-\epsilon}^\ell & \phi_{-\frac{3}{2}-\epsilon}^\ell & \cdots & \phi_{-\bar{\ell}_{1}+\frac{1}{2}-\epsilon}^\ell&\bar{\phi}_{-\frac{1}{2}-\epsilon}^{n-\ell+1} & \bar{\phi}_{-\frac{3}{2}-\epsilon}^{n-\ell+1} & \cdots & \bar{\phi}_{-\bar{\ell}_{2}+\frac{1}{2}-\epsilon}^{n-\ell+1} \\
\end{pmatrix},
\end{equation}
 recalling that
\begin{equation*}
\bar{\ell}_1:=
\begin{cases}
  \frac{\ell_1}{2}m+\ell_2 & \mbox{if } \ell_1 \textrm{ is even} \\
  \frac{\ell_1+1}{2}m & \mbox{if } \ell_1 \textrm{ is odd}
\end{cases},\qquad
\bar{\ell}_2:=
\begin{cases}
  \frac{\ell_1}{2}m & \mbox{if } \ell_1 \textrm{ is even}\\
  \frac{\ell_1-1}{2}m+\ell_2 & \mbox{if } \ell_1 \textrm{ is odd}
\end{cases},
\end{equation*}
and $\ell_1,\ell_2$ are the (unique) nonnegative integers such that
\[\ell=\ell_1m+\ell_2\quad\text{and}\quad 1\leq \ell_2\leq m.\]

When $n$ is even, we also define  $\bar{B}_{\frac{n}{2},\bm Z}$  to be the determinant obtained by replacing the $\frac{n}{2}$-th row in $B_{\frac{n}{2},\bm Z}$ by
   $$(\phi_{-\frac{1}{2}-\epsilon}^{\frac{n}{2}+1} , \phi_{-\frac{3}{2}-\epsilon}^{\frac{n}{2}+1} , \ldots , \phi_{-\bar{\ell}_1+\frac{1}{2}-\epsilon}^{ \frac{n}{2}+1},\bar{\phi}_{-\frac{1}{2}-\epsilon}^{\frac{n}{2}} ,\bar{\phi}_{-\frac{3}{2}-\epsilon}^{\frac{n}{2}} , \ldots  , \bar{\phi}_{-\bar{\ell}_2+\frac{1}{2}-\epsilon}^{\frac{n}{2}}).$$

\begin{lemt}\label{le:dualpair2le1}
Let $1\leq i,j\leq m$, $a\in\mathbb{Z}$ and  $1\leq \ell\leq \lfloor\frac{n}{2}\rfloor$. Assume that either $a>0$ or $a=0$ and $i<j$.
Then we have
\begin{equation}\label{eq:dualpair2eq1}
[\sum_{k=1}^{n}\psi_i^k(a-r)\bar{\psi}_j^k(r),B_{\ell,\bm Z}]=
\begin{cases}
  B_{\ell,\bm Z}^{r,i,j}+\bar{B}_{\ell,\bm Z}^{r,i,j} & \mbox{if } \alpha(r,a,i)\in I_{\bar{\ell}_2},~ \beta(r,0,j)\in I_{\bar{\ell}_1}  \\
  \bar{B}_{\ell,\bm Z}^{r,i,j}& \mbox{if }\alpha(r,a,i)\in I_{\bar{\ell}_2},~\beta(r,0,j)\notin I_{\bar{\ell}_1}\\
  B_{\ell,\bm Z}^{r,i,j}& \mbox{if }\alpha(r,a,i)\notin I_{\bar{\ell}_2},~\beta(r,0,j)\in I_{\bar{\ell}_1}\\
  0 & \mbox{otherwise}
\end{cases}.
\end{equation}
Here, when $\beta(r,0,j)\in I_{\bar{\ell}_1}$, $B_{\ell,\bm Z}^{r,i,j}$ is the determinant obtained by replacing the $\beta(r,0,j)$-column in $B_{\ell,\bm Z}$ by the column vector
  $$(\phi_{\alpha(r,a,i)-\frac{1}{2}-\epsilon}^1,\phi_{\alpha(r,a,i)-\frac{1}{2}-\epsilon}^2,\ldots,
  \phi_{\alpha(r,a,i)-\frac{1}{2}-\epsilon}^\ell)^{t},$$
  and, when $\alpha(r,a,i)\in I_{\bar{\ell}_2}$,
  $\bar{B}_{\ell,\bm Z}^{r,i,j}$ is the determinant obtained by replacing the $(\bar{\ell}_1+\alpha(r,a,i))$-column in $B_{\ell,\bm Z}$ by the column vector
  $$(-\bar{\phi}_{\beta(r,0,j)-\frac{1}{2}-\epsilon}^1,-\bar{\phi}_{\beta(r,0,j)-\frac{1}{2}-\epsilon}^2,
  \ldots,-\bar{\phi}
  _{\beta(r,0,j)-\frac{1}{2}-\epsilon}^\ell)^{t}.$$

Furthermore, for the case that $\beta(r,0,j)\in I_{\bar{\ell}_1}$ we have
 \begin{eqnarray}\label{eq:dualpair2eq3}
[\bar{B}_{\frac{n}{2},\bm Z}, B_{\ell^{\prime\prime},\bm Z}^{r,i,j}]=0\ (\text{with $n$ even}),\quad [B_{\ell',\bm Z}, B_{\ell,\bm Z}^{r,i,j}]=0\quad\text{and}\quad
 B_{\ell,\bm Z}^{r,i,j}|0\rangle=0,
\end{eqnarray}
and for the case that $\alpha(r,a,i)\in I_{\bar{\ell}_2}$ we have
 \begin{eqnarray}\label{eq:dualpair2eq2}
[\bar{B}_{\frac{n}{2},\bm Z}, \bar{B}_{\ell^{\prime\prime},\bm Z}^{r,i,j}]=0\ (\text{with $n$ even}),\quad
[B_{\ell',\bm Z}, \bar{B}_{\ell,\bm Z}^{r,i,j}]=0\quad\text{and}\quad
 \bar{B}_{\ell,\bm Z}^{r,i,j}|0\rangle=0,
\end{eqnarray}
where $\ell',\ell'',=1,2,\dots,n$ with $\ell'+\ell\leq n$ and $\ell^{\prime\prime}<\frac{n}{2}$.
\end{lemt}
\begin{proof}
The verification of  \eqref{eq:dualpair2eq1} is straightforward.
  For the equations \eqref{eq:dualpair2eq3}
 and \eqref{eq:dualpair2eq2},
we can conclude from \eqref{eq:commu} that $[\bar{B}_{\frac{n}{2},\bm Z}, B_{\ell^{\prime\prime},\bm Z}^{r,i,j}]=[\bar{B}_{\frac{n}{2},\bm Z}, \bar{B}_{\ell^{\prime\prime},\bm Z}^{r,i,j}]=0$ if $n$ is even and $\ell''<\frac{n}{2}$, and
that $[B_{\ell',\bm Z}, \bar{B}_{\ell,\bm Z}^{r,i,j}]=0=[B_{\ell',\bm Z}, B_{\ell,\bm Z}^{r,i,j}]=0$.
Furthermore, if $\beta(r,0,j)\in I_{\bar{\ell}_1}$, then
  $$\alpha(r,a,i)-\frac{1}{2}-\epsilon+\beta(r,0,j)-\frac{1}{2}-\epsilon=am+j-i-2\epsilon\geq 1-2\epsilon.$$
Using this, we get that
  $$\alpha(r,a,i)-\frac{1}{2}-\epsilon \geq 1-2\epsilon-(\bar{\ell}_1-\epsilon-\frac{1}{2})\geq -\bar{\ell}_1+\frac{3}{2}-\epsilon, $$
  which implies $B_{\ell,\bm Z}^{r,i,j}|0\rangle=0$.
Similarly, if $\alpha(r,a,i)\in I_{\bar{\ell}_2}$,
  then
  $$\beta(r,0,j)-\frac{1}{2}-\epsilon \geq 1-2\epsilon-(\bar{\ell}_2-\epsilon+\frac{1}{2})\geq -\bar{\ell}_2+\frac{3}{2}-\epsilon, $$
  which implies $\bar{B}_{\ell,\bm Z}^{r,i,j}|0\rangle=0$.
\end{proof}
\begin{lemt}\label{le:dualpair2le2}
Let $1\leq i,j\leq m$, $a\in\mathbb{Z}$, $r\in\bm Z$ and $\lfloor\frac{n}{2}\rfloor<\ell\leq n$. Assume that either $a>0$ or $a=0$ and $i<j$. Then we have
 \begin{eqnarray}\label{eq:dualpair2eq4}
 \sum_{k=1}^{n}\psi_i^k(a-r)\bar{\psi}_j^k(r)B_{\ell,\bm Z}|0\rangle=0.
 \end{eqnarray}
\end{lemt}
\begin{proof}
It is clear that   \eqref{eq:dualpair2eq4} holds when either $ \alpha(r,a,i)\notin I_{\bar{\ell}_2}$ or $\beta(r,0,j)\notin I_{\bar{\ell}_1}$.
Assume now that $\alpha(r,a,i)\in I_{\bar{\ell}_2}$ and $\beta(r,0,j)\in I_{\bar{\ell}_1}$.
For $1\leq i_1,i_2,j_1,j_2\leq \ell$,
 define $B_{i_1,j_1}$ to be the determinant  obtained by deleting the $i_1$-th row and $j_1$-th column of $B_{\ell,\bm Z}$, and
  define $B_{i_1,j_1}^{i_2,j_2}$ to be the determination  obtained by deleting the $i_1$-th, $i_2$-th rows and $j_1$-th, $j_2$-th columns of $B_{\ell,\bm Z}$.
  With the above notations, we obtain
 \begin{equation*}
  \begin{split}
 &\sum_{k=1}^n\psi_i^k(a-r)\bar{\psi}_j^k(r)B_{\ell,\bm Z}|0\rangle\\
=&\sum_{k=1}^\ell\phi_{\alpha(r,a,i)-\frac{1}{2}-\epsilon}^k\bar{\phi}_{\beta(r,0,j)-\frac{1}{2}-\epsilon}^k(\sum_{s=1}^{\ell}
(-1)^{s+\beta(r,0,j)}\phi_{-\beta(r,0,j)+\frac{1}{2}-\epsilon}^{s}B_{s,\beta(r,0,j)}|0\rangle\\
=&\sum_{k=1}^{\ell}(-1)^{k+\beta(r,0,j)}\phi_{\alpha(r,a,i)-\frac{1}{2}-\epsilon}^kB_{k,\beta(r,0,j)}|0\rangle\\
 =&\sum_{k=1}^\ell(-1)^{k+\beta(r,0,j)}\phi_{\alpha(r,a,i)-\frac{1}{2}-\epsilon}^k(\sum_{s=1}^{k-1}
 (-1)^{s+\bar{\ell}_1+\alpha(r,a,i)-1}\bar{\phi}_{-\alpha(r,a,i)+\frac{1}{2}-\epsilon}^{n-s+1}
 B_{k,\beta(r,0,j)}^{s,\alpha(r,a,i)+\bar{\ell}_1}\\
 &+\sum_{s=k+1}^{\ell}(-1)^{s+\bar{\ell}_1+\alpha(r,a,i)}
 \bar{\phi}_{-\alpha(r,a,i)+\frac{1}{2}-\epsilon}^{n-s+1}
 B_{k,\beta(r,0,j)}^{s,\alpha(r,a,i)+\bar{\ell}_1})|0\rangle\\
=&\sum_{k=n-\lfloor \frac{n}{2}\rfloor+1}^{\ell}(-1)^{n+\alpha(r,a,i)+\beta(r,0,j)+\bar{\ell}_1+1}
B_{k,\beta(r,0,j)}^{n-k+1,\alpha(r,a,i)+\bar{\ell}_1}|0\rangle\\
&+\sum_{k=n-\ell+1}^{\lfloor \frac{n}{2}\rfloor}(-1)^{n+\alpha(r,a,i)+\beta(r,0,j)+\bar{\ell}_1}
B_{k,\beta(r,0,j)}^{n-k+1,\alpha(r,a,i)+\bar{\ell}_1}|0\rangle=0.\\
  \end{split}
  \end{equation*}
\end{proof}

 Similar to Lemma \ref{le:dualpair2le1} and Lemma \ref{le:dualpair2le2}, we have the following results.

\begin{lemt}\label{le:dualpair2le3}
Let $1\leq i,j\leq m$, $a\in\mathbb{Z}$ and  $1\leq \ell\leq \lfloor\frac{n}{2}\rfloor$.
Assume that either $a>0$ or $a=0$ and $\bm Z=\mathbb{Z}+\frac{1}{2}$.
Then we have
\begin{equation}
[\sum_{k=1}^{n}\psi_i^k(a-r)\psi_j^{n+1-k}(r),B_{\ell,\bm Z}]=
\begin{cases}
  C_{\ell,\bm Z}^{r,i,j}+ \bar{C}_{\ell,\bm Z}^{r,i,j}& \mbox{if } \alpha(-r,0,j)\in I_{\bar{\ell}_2},~ \alpha(r,a,i)\in I_{\bar{\ell}_2}  \\
   C_{\ell,\bm Z}^{r,i,j}& \mbox{if }   \alpha(-r,0,j)\in I_{\bar{\ell}_2},~\alpha(r,a,i)\notin I_{\bar{\ell}_2}\\
  \bar{C}_{\ell,\bm Z}^{r,i,j}& \mbox{if }   \alpha(-r,0,j)\notin I_{\bar{\ell}_2},~\alpha(r,a,i)\in I_{\bar{\ell}_2}\\
  0 & \mbox{otherwise}
\end{cases}.
\end{equation}
Here, when $\alpha(-r,0,j)\in I_{\bar{\ell}_2}$, $C_{\ell,\bm Z}^{r,i,j}$ is the determinant obtained by replacing the $(\bar{\ell}_1+\alpha(-r,0,j))$-column in $B_{\ell,\bm Z}$ by the column vector
  $$(-\phi_{\alpha(r,a,i)-\frac{1}{2}-\epsilon}^1,-\phi_{\alpha(r,a,i)-\frac{1}{2}-\epsilon}^2,\ldots,
  -\phi_{\alpha(r,a,i)-\frac{1}{2}-\epsilon}^\ell)^{t},$$
  and, when $\alpha(r,a,i)\in I_{\bar{\ell}_2}$,
  $\bar{C}_{\ell,\bm Z}^{r,i,j}$ is the determinant obtained by replacing the $(\bar{\ell}_1+\alpha(r,a,i))$-column in $B_{\ell,\bm Z}$ by the column vector
  $$(-\phi_{\alpha(-r,0,j)-\frac{1}{2}-\epsilon}^1,-\phi_{\alpha(-r,0,j)-\frac{1}{2}-\epsilon}^2,
  \ldots,-\phi
  _{\alpha(-r,0,j)-\frac{1}{2}-\epsilon}^\ell)^{t}.$$

Furthermore, for the case that $\alpha(-r,0,j)\in I_{\bar{\ell}_2}$ we have
\begin{eqnarray}
[\bar{B}_{\frac{n}{2},\bm Z}, C_{\ell^{\prime\prime},\bm Z}^{r,i,j}]=0\ (\text{with $n$ even}),\quad [B_{\ell',\bm Z}, C_{\ell,\bm Z}^{r,i,j}]=0\quad\text{and}\quad
 C_{\ell,\bm Z}^{r,i,j}|0\rangle=0,
\end{eqnarray}
and for the case that $\alpha(r,a,i)\in I_{\bar{\ell}_2}$ we have
\begin{eqnarray}
[\bar{B}_{\frac{n}{2},\bm Z}, \bar{C}_{\ell^{\prime\prime},\bm Z}^{r,i,j}]=0\ (\text{with $n$ even}),\quad
[B_{\ell',\bm Z}, \bar{C}_{\ell,\bm Z}^{r,i,j}]=0\quad\text{and}\quad
 \bar{C}_{\ell,\bm Z}^{r,i,j}|0\rangle=0,
\end{eqnarray}
where $\ell'=1,2,\dots,n$ with $\ell'+ \ell\le n$ and $\ell''<\frac{n}{2}$.
\end{lemt}
\begin{lemt}\label{le:dualpair2le4}
Let $1\leq i,j\leq m$, $a\in\mathbb{Z}$, $r\in\bm Z$ and $\lfloor\frac{n}{2}\rfloor<\ell\leq n$. Assume that either $a>0$ or $a=0$ and $\bm Z=\mathbb{Z}+\frac{1}{2}$.
Then we have
 \begin{eqnarray}
 \sum_{k=1}^{n}\psi_i^k(a-r)\psi_j^{n+1-k}(r)B_{\ell,\bm Z}|0\rangle=0.
 \end{eqnarray}
\end{lemt}

\begin{lemt}\label{le:dualpair2le5}
Let $1\leq i,j\leq m$, $a\in\mathbb{Z}$ and $1\leq \ell\leq \lfloor\frac{n}{2}\rfloor$. Assume that either $a>0$ or $a=0$ and $\bm Z=\mathbb{Z}$.
Then we have
\begin{equation}
[\sum_{k=1}^{n}\bar{\psi}_i^k(a-r)\bar{\psi}_j^{n+1-k}(r),B_{\ell,\bm Z}]=
\begin{cases}
  D_{\ell,\bm Z}^{r,i,j}+ \bar{D}_{\ell,\bm Z}^{r,i,j}& \mbox{if } \beta(r,0,j)\in I_{\bar{\ell}_1},~\beta(-r,a,i)\in I_{\bar{\ell}_1}  \\
  D_{\ell,\bm Z}^{r,i,j}& \mbox{if }  \beta(r,0,j)\in I_{\bar{\ell}_1},~\beta(-r,a,i)\notin I_{\bar{\ell}_1}\\
  \bar{D}_{\ell,\bm Z}^{r,i,j}& \mbox{if }  \beta(r,0,j)\notin I_{\bar{\ell}_1},~\beta(-r,a,i)\in I_{\bar{\ell}_1}\\
  0 & \mbox{otherwise}
\end{cases}.
\end{equation}
Here, when $\beta(r,0,j)\in I_{\bar{\ell}_1}$, $D_{\ell,\bm Z}^{r,i,j}$ is the determinant obtained by replacing the $\beta(r,0,j)$-column in $B_{\ell,\bm Z}$ by the column vector
  $$(\bar{\phi}_{\beta(-r,a,i)-\frac{1}{2}-\epsilon}^1,\bar{\phi}_{\beta(-r,a,i)-\frac{1}{2}-\epsilon}^2,\dots,
  \bar{\phi}_{\beta(-r,a,i)-\frac{1}{2}-\epsilon}^\ell)^{t},$$
  and, when $\beta(-r,a,i)\in I_{\bar{\ell}_1}$,
  $\bar{D}_{\ell,\bm Z}^{r,i,j}$ is the determinant obtained by replacing the $\beta(-r,a,i)$-column in $B_{\ell,\bm Z}$ by the column vector
  $$(\bar{\phi}_{\beta(r,0,j)-\frac{1}{2}-\epsilon}^1,\bar{\phi}_{\beta(r,0,j)-\frac{1}{2}-\epsilon}^2,\dots,
  \bar{\phi}_{\beta(r,0,j)-\frac{1}{2}-\epsilon}^\ell)^{t}.$$

Furthermore, for the case $\beta(r,0,j)\in I_{\bar{\ell}_1}$ we have
 \begin{eqnarray}
 [\bar{B}_{\frac{n}{2},\bm Z}, D_{\ell^{\prime\prime},\bm Z}^{r,i,j}]=0\ (\text{with $n$ even}),
[B_{\ell',\bm Z}, D_{\ell,\bm Z}^{r,i,j}]=0\quad\text{and}\quad
D_{\ell,\bm Z}^{r,i,j}|0\rangle=0,
\end{eqnarray}
and for the case $\beta(-r,a,i)\in I_{\bar{\ell}_1}$ we have
\begin{eqnarray}
[\bar{B}_{\frac{n}{2},\bm Z}, D_{\ell^{\prime\prime},\bm Z}^{r,i,j}]=0\ (\text{with $n$ even}),
[B_{\ell',\bm Z}, \bar{D}_{\ell,\bm Z}^{r,i,j}]=0\quad\text{and}\quad
 \bar{D}_{\ell,\bm Z}^{r,i,j}|0\rangle=0,
\end{eqnarray}
where $\ell'=1,2,\dots,n$ with $\ell'+ \ell\leq n$ and $\ell''<\frac{n}{2}$.
\end{lemt}
\begin{lemt}\label{le:dualpair2le6}
Let $1\leq i,j\leq m$, $a\in\mathbb{Z}$ and $\lfloor\frac{n}{2}\rfloor< \ell\leq n$. Assume that either $a>0$ or $a=0$ and $\bm Z=\mathbb{Z}$. Then we have
 \begin{eqnarray}
 \sum_{k=1}^{n}\bar{\psi}_i^k(a-r)\bar{\psi}_j^{n+1-k}(r)B_{\ell,\bm Z}|0\rangle=0.
 \end{eqnarray}
\end{lemt}

The following result can be also checked directly.

\begin{lemt}\label{le:dualpair2le7}
Let $1\leq p\leq n-1$ and $1\le \ell\le n$. Then we have
\begin{equation*}
  [\sum_{k=1}^{m}\sum_{r\in\bm Z}(:\psi_k^p(-r)\bar{\psi}_k^{p+1}(r):-:\psi_k^{n-p}(-r)\bar{\psi}_k^{n+1-p}(r):),B_{\ell,\bm Z}]=0.
\end{equation*}
\end{lemt}

For $\ell=1,2,\dots,n$, define  $\bm{\mu}_\ell$  and $\bm{\mu}^{\prime}_\ell$ in $\mathcal{R}(\mathrm{O}_n)$ by

\[\bm{\mu}_\ell:=(\underbrace{1,1,\ldots,1}_{\ell},0,0,\dots,0)\quad\text{and}\quad\bm{\mu}^{\prime}_\ell:
=(\underbrace{1,1,\ldots,1}_{min\{\ell,n-\ell\}},0,0,\dots,0).\]

\begin{lemt}\label{le:dualpair2le8}
For $\ell=1,2,\dots,n$ and $h\in \mathrm{H}_{\mathrm{SO}_n}$, we have
\begin{equation*}
  h.B_{\ell,\bm Z}|0\rangle=h^{\bm{\mu}^{\prime}_\ell} B_{\ell,\bm Z}|0\rangle.
\end{equation*}
\end{lemt}
\begin{proof}
For $1\leq p \leq \lfloor \frac{n}{2}\rfloor$, as operators on $\mathcal{F}_N(\bm{Z})$, we have
 $$e_{p,p}=\sum_{k=1}^{m}\sum_{r\in\bm Z}:\psi_k^p(-r)\bar{\psi}_k^{p}(r):-:\psi_k^{n+1-p}(-r)\bar{\psi}_k^{n+1-p}(r):.$$
Using this, one can get that
   \begin{equation*}
   e_{p,p}B_{\ell,\bm Z}=
     \begin{cases}
       B_{\ell,\bm Z}e_{p,p}+B_{\ell,\bm Z} & \mbox{if } p\leq \ell\textrm{ and } p\leq n-\ell \\
       B_{\ell,\bm Z}e_{p,p} &\mbox{otherwise }
     \end{cases},
   \end{equation*}
  which implies the lemma.
\end{proof}

\begin{lemt}\label{le:dualpair2le9}
For $\ell=1,2,\dots,n$ and $x\in \widehat{\mathfrak{sp}}_m(\mathbb{C}_q)_0$, we have
\begin{equation*}
x.B_{\ell,\bm Z}|0\rangle=\eta_{\bm{\mu}_\ell,\bm{Z}}(x) B_{\ell,\bm Z}|0\rangle.
\end{equation*}
\end{lemt}
\begin{proof}
Recall that $\widehat{\mathfrak{sp}}_{2m}(\mathbb{C}_q)_0$ is spanned by the elements $\bm{c}$ and $f_{i,i}(0,b)$ for $1\le i\le m$ and $b\in \Z$.
And, as operator on $\mathcal{F}_N(\bm{Z})$, we have
$$f_{i,i}(0,b)=\sum_{r\in\bm Z}q^{-br}\sum_{k=1}^n:\psi_i^k(-r)\bar{\psi}_i^k(r):-n\omega_{\bm Z}(b).$$
Observe that
   \begin{equation*}
   \sum_{k=1}^n\psi_i^k(-r)\bar{\psi}_i^k(r)B_{\ell,\bm Z}|0\rangle=
  \begin{cases}
    B_{\ell,\bm Z}|0\rangle & \mbox{if }  \beta(r,0,i)\in I_{\bar{\ell}_1} \\
    -B_{\ell,\bm Z}|0\rangle & \mbox{if } \alpha(r,0,i) \in I_{\bar{\ell}_2} \\
    0 & \mbox{otherwise}
  \end{cases}.
  \end{equation*}
  Then  $x.B_{\ell,\bm Z}|0\rangle=\eta_{\bm{\mu}_\ell,\bm{Z}}(x) B_{\ell,\bm Z}|0\rangle$ follows immediately.
\end{proof}

   For every $\bm\mu=(\mu_1,\mu_2,\dots,\mu_n)\in \mathcal{R}(\mathrm{O}_n)$
we set
   $$v_{\bm\mu,\bm Z}:=B_{1,\bm Z}^{\mu_1-\mu_2}B_{2,\bm Z}^{\mu_2-\mu_3}\cdots B_{n,\bm Z}^{\mu_n-\mu_{n+1}}|0\rangle\in\mathcal{F}_N(\bm Z),$$
   where $\mu_{n+1}=0$.
  If $n$ is even and $d(\bm\mu)=\frac{n}{2}$,  we also set
   $$\bar{v}_{\bm\mu,\bm Z}:=B_{1,\bm Z}^{\mu_1-\mu_2}B_{2,\bm Z}^{\mu_2-\mu_3}\dots \bar{B}_{ \frac{n}{2},\bm Z}^{\mu_ \frac{n}{2}}|0\rangle\in\mathcal{F}_N(\bm Z),$$
   recalling  $\overline{\bm\mu}=(\mu_1,\mu_2,\dots,\mu_{\frac{n}{2}-1},-\mu_{\frac{n}{2}})$.

%

\begin{prpt}\label{pr:dualpair3}
Let $\bm\mu\in \mathcal{R}(\mathrm{O}_n)$.
 If $d(\bm\mu)<\frac{n}{2}$, then
\begin{itemize}
\item $v_{\bm\mu,\bm Z}$ is an  $(\mathrm{SO}_n,\widehat{\mathfrak{sp}}_{2m}(\mathbb{C}_q))$-highest weight vector in $\mathcal{F}_N(\bm Z)$ with highest weight $(\bm\mu,\eta_{\bm\mu,\bm Z})$;
    \item $v_{\tilde{\bm\mu},\bm Z}$ is an $(\mathrm{SO}_n,\widehat{\mathfrak{sp}}_{2m}(\mathbb{C}_q))$-highest weight vector in $\mathcal{F}_N(\bm Z)$ with highest weight ($\bm\mu,\eta_{\tilde{\bm\mu},\bm Z}$).
    \end{itemize}
And, if $d(\bm\mu)=\frac{n}{2}$ (and so $n$ is even), then
\begin{itemize}
\item $v_{\bm\mu,\bm Z}$ is an $(\mathrm{SO}_n,\widehat{\mathfrak{sp}}_{2m}(\mathbb{C}_q))$-highest weight vector in $\mathcal{F}_N(\bm Z)$ with
    highest weight ($\bm\mu,\eta_{\bm\mu,\bm Z}$);
    \item $\bar{v}_{\bm\mu,\bm Z}$ is an $(\mathrm{SO}_n,\widehat{\mathfrak{sp}}_{2m}(\mathbb{C}_q))$-highest weight vector in $\mathcal{F}_N(\bm Z)$ with
    highest weight  ($\overline{\bm\mu},\eta_{\bm\mu,\bm Z}$).
\end{itemize}
  \end{prpt}
  \begin{proof} Note that the positive parts   $\widehat{\mathfrak{sp}}_{2m}(\mathbb{C}_q)_{+,\Z}$ and
   $\widehat{\mathfrak{sp}}_{2m}(\mathbb{C}_q)_{+,\Z+\frac{1}{2}}$ of $\widehat{\mathfrak{sp}}_{2m}(\mathbb{C}_q)$ are different.
And,
  the action of $\widehat{\mathfrak{sp}}_{2m}(\mathbb{C}_q)_{+,\bm{Z}}$  on $\mathcal{F}_N(\bm{Z})$ is as follows:
\begin{itemize}
\item
$f_{i,j}t_0^at_1^b=\sum_{r\in\bm Z}q^{-br}\left(\sum_{k=1}^n:\psi_i^k(a-r)\bar{\psi}_j^k(r):\right)
$ for $1\le i,j\le m$ and $a,b\in \Z$ such that either $a>0$ or $a=0$ and $i<j$;
\item
$g_{i,j}t_0^at_1^b=\sum_{r\in\bm Z}q^{-br}\left(\sum_{k=1}^n:\psi_i^k(a-r)\psi_j^{n+1-k}(r):\right)$ for $1\le i,j\le m$ and $a,b\in \Z$ such that either $a>0$ or $a=0$ and $\bm Z=\mathbb{Z}+\frac{1}{2}$;
\item
$h_{i,j}t_0^at_1^b=\sum_{r\in\bm Z}q^{-br}\left(\sum_{k=1}^n:\bar{\psi}_i^k(a-r)\bar{\psi}_j^{n+1-k}(r):\right)
$
for $1\le i,j\le m$ and $a,b\in \Z$ such that either $a>0$ or $a=0$ and $\bm Z=\mathbb{Z}$.
\end{itemize}

  If $\bm\mu\in \mathcal{R}(\mathrm{O}_n)$ with $d(\bm\mu)<\frac{n}{2}$, then we have
  $$v_{\tilde{\bm\mu},\bm Z}=B_{1,\bm Z}^{\mu_1-\mu_2}B_{2,\bm Z}^{\mu_2-\mu_3}\cdots B_{p,\bm Z}^{\mu_p-1}B_{n-d(\bm\mu),\bm Z}|0\rangle.$$
From Lemmas \ref{le:dualpair2le1}-\ref{le:dualpair2le6}, one can conclude that
  \[\widehat{\mathfrak{sp}}_{2m}(\mathbb{C}_q)_{+,\bm Z}v_{\bm\mu,\bm Z}=0\quad\textrm{and}\quad\widehat{\mathfrak{sp}}_{2m}(\mathbb{C}_q)_{+,\bm Z}v_{\tilde{\bm\mu},\bm Z}=0. \]
Meanwhile, it follows from Lemmas \ref{le:dualpair2le7},  \ref{le:dualpair2le8} and \ref{le:dualpair2le9}  that
 \begin{eqnarray*}
 && \mathrm{N}_{\mathrm{SO}_n}^{+}v_{\bm\mu,\bm Z}=v_{\bm\mu,\bm Z},\quad \mathrm{N}_{\mathrm{SO}_n}^{+}v_{\tilde{\bm\mu},\bm Z}=v_{\tilde{\bm\mu},\bm Z},\quad
x.v_{\bm\mu,\bm Z}=\eta_{\bm\mu,\bm Z}(x)v_{\bm\mu,\bm Z},\\
&& h.v_{\bm\mu,\bm Z}=h^{\bm\mu}v_{\bm\mu,\bm Z},\quad x.v_{\tilde{\bm\mu},\bm Z}=\eta_{\tilde{\bm\mu},\bm Z}(x)v_{\bm\mu,\bm Z},\quad h.v_{\tilde{\bm\mu},\bm Z}=h^{\tilde{\bm\mu}}v_{\bm\mu,\bm Z}
  \end{eqnarray*}
  for $x\in \widehat{\mathfrak{sp}}_{2m}(\mathbb{C}_q)_0$ and $h\in \mathrm{H}_{\mathrm{SO}_n}$.
  This proves the first two assertions in proposition.

    If $n$ is even and $\bm\mu\in \mathcal{R}(\mathrm{O}_n)$ with $d(\bm\mu)=\frac{n}{2}$,
    then it follows from Lemmas \ref{le:dualpair2le1}-\ref{le:dualpair2le9} that $v_{\bm\mu,\bm Z}$ is an $(\mathrm{SO}_n,\widehat{\mathfrak{sp}}_{2m}(\mathbb{C}_q))$-highest weight vector in $\mathcal{F}_N(\bm Z)$ with highest weight ($\bm\mu,\eta_{\bm\mu,\bm Z}$). In particular, $B_{\frac{n}{2},\bm Z}|0\rangle$ is an $(\mathrm{SO}_n,\widehat{\mathfrak{sp}}_{2m}(\mathbb{C}_q))$-highest weight vector with highest weight ($\bm\mu_{\frac{n}{2}},\eta_{\bm\mu_{\frac{n}{2}},\bm Z}$).
Note that from \eqref{eq:glnmodiso} we have
  $\tau.B_{\frac{n}{2},\bm Z}|0\rangle=\bar{B}_{\frac{n}{2},\bm Z}|0\rangle$,
   recalling that $\tau$ denotes the $n\times n$-matrix which interchanges the $\frac{n}{2}$-th row and $(\frac{n}{2}+1)$-th row of the identity matrix $I_n$.
     Then by Proposition \ref{prop:dualpairs} we have
  \begin{eqnarray*}&&y.\bar{B}_{\frac{n}{2},\bm Z}|0\rangle=y.\tau.B_{\frac{n}{2},\bm Z}|0\rangle=\tau.y.B_{\frac{n}{2},\bm Z}|0\rangle=0,\\
 &&x.\bar{B}_{\frac{n}{2},\bm Z}|0\rangle=x\tau.B_{\frac{n}{2},\bm Z}|0\rangle=\tau x.B_{\frac{n}{2},\bm Z}|0\rangle=\eta_{\bm\mu, \bm Z}(x)\bar{B}_{\frac{n}{2},\bm Z}|0\rangle,\\
 &&h.\bar{B}_{\frac{n}{2},\bm Z}|0\rangle=h\tau.B_{\frac{n}{2},\bm Z}|0\rangle=\tau(\tau^{-1}h\tau).B_{\frac{n}{2},\bm Z}|0\rangle=(\tau^{-1}h\tau)^{\bm\mu}\bar{B}_{\frac{n}{2},\bm Z}|0\rangle=(h)^{\overline{\bm\mu}}\bar{B}_{\frac{n}{2},\bm Z}|0\rangle
 \end{eqnarray*}
  for $y\in \widehat{\mathfrak{sp}}_{2m}(\mathbb{C}_q)_{+,\bm Z}$, $x\in \widehat{\mathfrak{sp}}_{2m}(\mathbb{C}_q)_0$ and $h\in\mathrm{H}_{\mathrm{SO}_n}$. This together with Lemmas \ref{le:dualpair2le1}-\ref{le:dualpair2le9} implies that $\bar{v}_{\bm\mu,\bm Z}$ is an $(\mathrm{SO}_n,\widehat{\mathfrak{sp}}_{2m}(\mathbb{C}_q))$-highest weight vector with highest weight ($\overline{\bm\mu},\eta_{\bm\mu,\bm Z}$). Thus we complete the proof.

\end{proof}

\textbf{Proof of Theorem \ref{thm:main}:}
In view of Proposition \ref{prop:howedual}, we have the decomposition
\begin{eqnarray*}
\mathcal{F}_N(\bm{Z})=\bigoplus_{\bm{\mu}\in \mathrm{Spec}_{\mathrm{G}}(\mathcal{F}_N(\bm{Z}))} L_{\mathrm{G}}(\bm\mu)\otimes
 \mathrm{Hom}_{\mathrm{G}}(L_{\mathrm{G}}(\bm\mu),\mathcal{F}_N(\bm{Z})),
\end{eqnarray*}
where  $\mathrm{Spec}_{\mathrm{G}}(\mathcal{F}_N(\bm{Z}))\subset \mathcal{R}(\mathrm{G})$
denotes the  spectrum of $\mathcal{F}_N(\bm{Z})$ as a locally regular $\mathrm{G}$-module.
From Propositions \ref{pr:dualpair1}, \ref{pr:dualpair2} and \ref{pr:dualpair3}, one can conclude that
that $\mathrm{Spec}_{\mathrm{G}}(\mathcal{F}_N(\bm{Z}))=\mathcal{R}(\mathrm{G})$.
 Furthermore, since
each multiplicity space $\mathrm{Hom}_{\mathrm{G}}(L_{\mathrm{G}}(\bm\mu),\mathcal{F}_N(\bm{Z}))$ is irreducible as a $\mathfrak{g}$-module,
we have $\mathrm{Hom}_{\mathrm{G}}(L_{\mathrm{G}}(\bm\mu),\mathcal{F}_N(\bm{Z}))=L_{\mathfrak{g},\bm{Z}}(\eta_{\bm{\mu},\bm{Z}})$.
This completes the proof of Theorem \ref{thm:main}.

\end{document}